\documentclass{amsart}

\usepackage[utf8]{inputenc}
\usepackage[english]{babel}

\usepackage{fullpage}

\usepackage{amsmath}
\usepackage{amsfonts}
\usepackage{amsthm}
\usepackage{amssymb}
\usepackage{theoremref}
\usepackage{colonequals}
\usepackage{graphicx}
\usepackage{cancel}
\usepackage{csquotes}
\usepackage{url}
\usepackage{subcaption}
\usepackage[inline]{enumitem}
\usepackage{nameref}

\newcounter{master}
\numberwithin{master}{section}

\theoremstyle{plain}

\newtheorem{theorem}[master]{Theorem}

\newtheorem{lemma}[master]{Lemma}

\newtheorem{conjecture}[master]{Conjecture}
\newtheorem{dummy}{\dummythmname}

\theoremstyle{definition}

\newtheorem{definition}[master]{Definition}
\newtheorem{claim}[master]{Claim}
\pretocmd{\endclaim}{\hfill$\blacktriangleleft$}{}{}

\theoremstyle{remark}

\newtheorem*{remark}{Remark}
\pretocmd{\endremark}{\hfill$\blacktriangleleft$}{}{}

\makeatletter
\let\c@equation\c@master

\let\c@figure\c@master

\let\c@table\c@master

\makeatother

\newenvironment{dualtheorem}[1]{\def\dummythmname{\thnameref{#1}}\def\thedummy{\ref{#1}*}\begin{dummy}}{\end{dummy}}

\makeatletter
\def\@cite#1#2{\textup{[\textbf{#1}\if@tempswa , #2\fi]}}
\def\@biblabel#1{[\textbf{#1}]}
\makeatother

\DeclareMathOperator{\diag}{diag}

\DeclareMathOperator{\Gr}{Gr}
\DeclareMathOperator{\Graff}{Graf{}f}
\DeclareMathOperator{\Int}{int}
\DeclareMathOperator{\PGA}{PGA}

\DeclareMathOperator{\SL}{SL}

\DeclareMathOperator{\sgn}{sgn}
\DeclareMathOperator{\Span}{span}

\makeatletter\def\tr#1{{\reset@font[}#1{\reset@font]}}\makeatother

\usepackage{tikz}
\usetikzlibrary{calc, decorations.pathmorphing, decorations.pathreplacing}

\tikzset{
    every path/.style={
        thick,
    },
    every node/.style={
        draw=black,
        fill=blue,
        thick,
        circle,
        minimum size=6pt,
        inner sep=0pt,
    },
    tag/.style={
        draw=none,
        fill=none,
    },
}

\def\backmatter{\def\subsection##1{\def\thesubsubsection{\S\ref{##1}}\subsubsection{\nameref{##1}}}}

\allowdisplaybreaks

\begin{document}

\title{On flat shadow boundaries from point light sources and the characterization of ellipsoids}

\author[B. Zawalski]{Bartłomiej Zawalski}
\address{Kent State University, Kent, OH, USA}
\email{bzawalsk@kent.edu}
\address{Case Western Reserve University, Cleveland, OH, USA}
\email{bartlomiej.zawalski@case.edu}
\thanks{The author is supported in part by U.S. National Science Foundation Grants DMS-1900008 and DMS-2247771.}

\begin{abstract}
In his classical work, W.~Blaschke proved that a convex body whose shadow boundaries are flat for every direction of parallel illumination must be an ellipsoid. An extension recently proposed by I.~Gonzalez-Garc\'ia, J.~Jer\'onimo-Castro, E.~Morales-Amaya, and D.J.~Verdusco-Hern\'andez predicts that the same conclusion holds for illumination by point light sources located on a hypersurface enclosing the body. We confirm this conjecture for convex bodies with sufficiently smooth boundaries. We further develop a duality framework relating illumination by point light sources to classical symmetry properties of hyperplane sections, linking several known and conjectured characterizations of quadrics from these complementary viewpoints.
\end{abstract}

\maketitle

\tableofcontents

%%%%%%%%%%%%%%%%%%%%%%%%%%%%%%%%%%%%%%%%%%%%%%%%%%%%%%%%%%%%%%%%%%%%%%
\section{Introduction}\label{sec:01}
%%%%%%%%%%%%%%%%%%%%%%%%%%%%%%%%%%%%%%%%%%%%%%%%%%%%%%%%%%%%%%%%%%%%%%

A classical problem in convex geometry is to determine global properties of a convex body $K\subset\mathbb R^n$, $n\geq 3$, from partial geometric information, such as its projections, sections, or shadows. A different, yet closely related, question arises from illumination problems: \emph{How is the geometry of $ K$ reflected in the structure of its shadow boundaries generated by point light sources?}\\

In his classical work, W.~Blaschke \cite[Anhang.VII]{BlaschkeKK} studied shadow boundaries for general convex bodies under illumination by parallel light sources. In this setting, shadow boundaries arise as points of tangency of supporting cylinders. He showed that if the shadow boundary of a convex body $K$ is flat for every direction of parallel illumination, then $K$ is an ellipsoid. Note that parallel illumination may be regarded as illumination from point light sources at infinity, while supporting cylinders correspond to supporting cones with apexes at infinity. Since tangency and flatness are projective invariants, the classical problem of Blaschke is equivalent to the problem of illumination by point light sources lying in a hyperplane disjoint from $K$.\\

In this spirit, A.~Marchaud \cite{Marchaud1959} investigated illumination by point light sources lying in a hyperplane $H$ that may intersect the convex body. He showed in dimension $3$ that if $H$ is tangent to the body at a single point, then the body must be an ellipsoid. If, however, $H$ intersects the interior of $K$, then, besides ellipsoids, other possibilities arise. Namely, the boundary of $K$ may consist of two caps of convex quadrics lying on opposite sides of $H$ and tangent along it, with at least one cap belonging to an ellipsoid, while the other may degenerate to a cone.\\

A natural next step is to replace the hyperplane by an arbitrary hypersurface $S$ enclosing $K$. This case was studied by J.~Jer\'onimo-Castro, L.~Montejano, and E.~Morales-Amaya, who showed that if $S$ is a polyhedral hypersurface enclosing $K$ and every point light source on $S$ creates a flat shadow boundary on $K$, then $K$ must be an ellipsoid \cite[Theorem~3.2]{flatgrazesconvexbodies}. As a consequence, the same conclusion holds if flat shadow boundaries are created by all point light sources belonging to a thin open shell enclosing $K$ \cite[Corollary~3.3]{flatgrazesconvexbodies}. In particular, it suffices to assume that this condition holds for all point light sources in some open neighborhood of $K$ \cite[Corollary~3.4]{flatgrazesconvexbodies}, thereby recovering an earlier theorem of G.R.~Burton \cite[Theorem~1]{Burton1977}.\\

Motivated by these developments, the following conjecture has emerged as a central open problem:

\begin{conjecture}[{cf. \cite[Conjecture~1]{Gonzalez_Garcia2022}, \cite[Conjecture~1]{https://doi.org/10.1112/mtk.12176}, \cite[Conjecture~3.1]{flatgrazesconvexbodies}}]\thlabel{con:01}
Let $K\subset\mathbb R^n$, $n\geq 3$, be a convex body and let $S$ be a hypersurface, which is the image of an embedding of the sphere $\mathbb S^{n-1}$, such that $K$ is contained in the interior of $S$. If any point light source on $S$ creates a flat shadow boundary on $K$, then $K$ is an ellipsoid.
\end{conjecture}

The conjecture asserts that ellipsoids are the only convex bodies whose shadow boundaries are flat under illumination from an enclosing hypersurface. In this paper, we confirm \thref{con:01} for convex bodies with boundary of class $C^3$. In fact, our approach yields an even stronger conclusion: it is enough to assume the flatness condition for only finitely many point light sources on every supporting hyperplane. To formulate our main result, we introduce the following notion:

\begin{definition}
A set containing at least $d$ points in a $d$-dimensional affine space is in a \emph{general linear position with respect to $p$} if no $d$ of them lie in a hyperplane passing through $p$.
\end{definition}

\begin{theorem}\thlabel{thm:01}
Let $K\subset\mathbb R^n$, $n\geq 4$, be a convex body with boundary of class $C^3$. Suppose that for every point $p\in\partial K$ on the boundary, there are at least $L(n)$ point light sources on the tangent hyperplane $T_p\partial K$ in a general linear position with respect to $p$, that create flat shadow boundaries on $K$. Then $K$ is an ellipsoid.
\end{theorem}

\noindent Here $L(n)$ is an explicit constant depending only on the dimension. The above statement is still true in dimension $3$, but for that, we need to assume additional smoothness:

\begin{theorem}\thlabel{thm:02}
Let $K\subset\mathbb R^3$ be a convex body with boundary of class $C^5$. Suppose that for every point $p\in\partial K$ on the boundary, there are at least $L(3)$ point light sources on the tangent plane $T_p\partial K$ in a general linear position with respect to $p$, that create flat shadow boundaries on $K$. Then $K$ is an ellipsoid.
\end{theorem}

Since the proofs for these cases are essentially different, we have decided to formulate them as separate theorems. The smoothness assumption, while inherent in our argument, seems to be superfluous, and we believe it can be relaxed, especially in \thref{thm:02}. The assumption that every hyperplane tangent to $K$ contains at least $L(n)$ point light sources is similar in spirit to the assumption of \cite[Theorem~2]{Bianchi1987}, but much weaker. Admittedly, we have been quite profligate in assuming that the set of light sources meets our stringent \enquote{generality} requirements. But otherwise we would surely sink into the mire of countless configurations and degenerate cases. Moreover, to eliminate some of them, we would need to assume additional smoothness, which is highly undesirable. However, in most applications, the light sources will be distributed on a convex hypersurface encompassing $K$ anyway. We have strong reasons to believe that the sharp constant should be $L(n)=n+1$, but we were unable to prove this in full generality.\\

\thref{thm:01} should also be viewed in the context of a local version of the classical theorem of Blaschke obtained by Jer\'onimo-Castro, Montejano, and Morales-Amaya \cite[Theorem~2.2]{flatgrazesconvexbodies}. They showed that if a convex body $K$ admits flat shadow boundaries for all light directions in an open subset of the Grassmannian, then these shadow boundaries coincide with those of a quadric. Remarkably, the proof of \thref{thm:01} reveals a different and, in fact, stronger type of locality. Namely, one can deduce the quadric structure of a single open patch of $\partial K$ from the flatness only of the portions of shadow boundaries that intersect that patch. This demonstrates that the geometry of $K$ is already determined by purely local information carried by finitely many shadow boundaries through each point on the boundary, rather than requiring flatness along entire shadow boundaries as in the earlier local results.

\begin{figure}
\includegraphics[width=.5\textwidth]{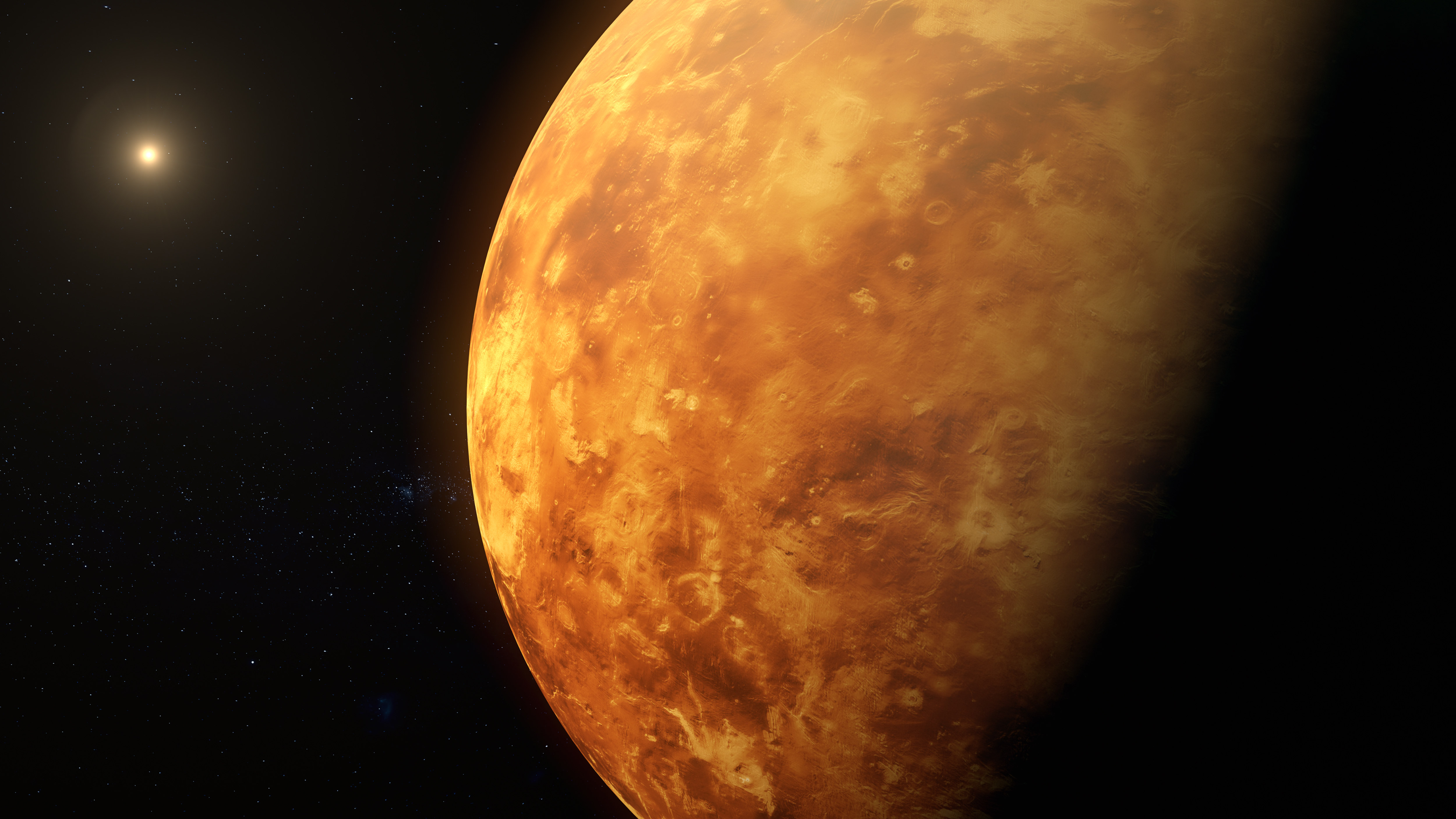}
\caption{The planet Venus illuminated by a point light source (the Sun), with the visible terminator representing the boundary between illuminated and shadowed regions. (Steven Molina/\protect\url{shutterstock.com})}
\label{fig:02}
\end{figure}

%%%%%%%%%%%%%%%%%%%%%%%%%%%%%%%%%%%%%%%%%%%%%%%%%%%%%%%%%%%%%%%%%%%%%%
\section{Duality between supporting cones and affine sections}
%%%%%%%%%%%%%%%%%%%%%%%%%%%%%%%%%%%%%%%%%%%%%%%%%%%%%%%%%%%%%%%%%%%%%%

The relation between supporting cones and affine sections is not only conceptual. Quite the contrary, for there is a very precise notion of duality between them, which allows us to transfer many theorems between these two worlds, and see them in a completely new light.\\

Denote by $\mathbb A^n$ the $n$-dimensional affine space. We complete $\mathbb A^n$ to $n$-dimensional projective space $\mathbb P(\mathbb A^n)$, by adding the \emph{hyperplane at infinity} and \emph{points at infinity} which lie on this hyperplane:
\begin{gather*}
\mathbb P(\mathbb A^n)\colonequals\mathbb A^n\sqcup\Gr_1(\mathbb R^n),\\
\Graff_{n-1}(\mathbb P(\mathbb A^n))\colonequals\{H\sqcup\Gr_1(\vec H):H\in\Graff_{n-1}(\mathbb A^n)\}\sqcup\{\Gr_1(\mathbb R^n)\}.
\end{gather*}
By choosing the origin $0\in\mathbb A^n$, we can identify the affine space $\mathbb A^n$ with its underlying linear space $\mathbb R^n$. For each point $p\in\mathbb A^n$ other than $0$, the projective hyperplane $H\colonequals\{x\in\mathbb A^n:\langle\boldsymbol x,\boldsymbol p\rangle=1\}\sqcup\Gr_1(\Span\{\boldsymbol p\}^\perp)$ is called the \emph{polar} of $p$, and the point $p$ is called the \emph{pole} of $H$. Further, we define the polar of $0$ to be the hyperplane at infinity, and the poles of the hyperplanes through $0$ to be the points at infinity, corresponding to their normal lines. The function $\varphi$ mapping points to their polars and hyperplanes to their poles preserves incidence, and hence is a \emph{correlation}.\\

Now, let $K\subset\mathbb R^n$, $n\geq 2$, be a convex body in $n$-dimensional affine space $\mathbb A^n$, containing the origin $0$ in its interior. We define the \emph{polar body} of $K$ by $K^\circ\colonequals\{x\in\mathbb A^n:\langle\boldsymbol x,\boldsymbol y\rangle\leq 1\text{ for every }y\in K\}$ (cf. \cite[\S 1.6.1]{Schneider_2013}). The polar body is again a convex body, containing the origin $0$ in its interior. Moreover, $K^{\circ\circ}=K$, so the correspondence is a true duality (cf. \cite[Theorem~1.6.1]{Schneider_2013}). In correlation $\varphi$, points on the boundary of $K$ correspond to hyperplanes supporting $K^\circ$. It follows that each section $\partial K\cap H$ of $\partial K$ by a hyperplane $H$, corresponds to the set of hyperplanes supporting $K^\circ$ and passing through $\varphi(H)\equalscolon H^\circ$, whose envelope is the supporting cone of $K^\circ$ with apex $H^\circ$. Note that supporting cones with apexes at infinity correspond to hyperplanar sections passing through the origin.\\

Further, if $A\in\PGA(\mathbb A^n)$ is a projective automorphism, then $(Ap)^\circ=A^\circ p^\circ$, where $\boldsymbol A^\circ=\boldsymbol A^{-\top}$, in terms of matrices. Hence, a different choice of the origin $0\in\mathbb A$ results in application of a certain projective automorphism $A\in\PGA(\mathbb A^n)$ to $K^\circ$, and the group of projective symmetries of $K$ is isomorphic to the group of projective symmetries of $K^\circ$. Moreover, the group of projective symmetries of any affine section $K\cap H$ is likewise isomorphic to the group of projective symmetries of the corresponding supporting cone of $K^\circ$ with apex $H^\circ$. The above observation establishes a duality between the characterizations of ellipsoids in terms of projective symmetries of their sections and in terms of projective symmetries of their supporting cones. However, note that both maps $A$ and $A^\circ$ are simultaneously affine if and only if they both preserve the origin. Indeed, a projective map is affine if and only if it preserves the hyperplane at infinity, whence its dual preserves the origin. Therefore, the groups of affine symmetries of $K$ and $K^\circ$ no longer need to be isomorphic, unless the origin is chosen at the Santal\'o point of $K$ \cite[(7.4.18)]{Schneider_2013}. Moreover, even the very shape of $K^\circ$ depends on the choice of the origin, whereas in the projective setting it does not. That is why projective symmetries seem to be more natural in this context.\\

Nevertheless, it is still worth considering affine symmetries. Observe that if a hyperplane $H$ does not pass through the origin, then every affine symmetry $A$ of $K\cap H$ can be uniquely extended to a linear automorphism $\boldsymbol A$ of the ambient space, which gives rise to a linear symmetry $\boldsymbol A^\circ$ of the corresponding supporting cone of $K^\circ$ with apex $H^\circ$. Moreover, since $A$ is conjugated to an orthogonal map, both $\boldsymbol A$ and $\boldsymbol A^\circ$ share the same characteristic polynomial. On the other hand, if a hyperplane $H$ passes through the origin, then every linear symmetry $\boldsymbol A$ of $K\cap H$ can likewise be extended to a linear automorphism $\boldsymbol A$ of the ambient space. Such extensions, which differ in the choice of the image of the normal vector to $H$, give rise to linear symmetries of the corresponding supporting cylinder of $K^\circ$ with axis $H^\circ$, which differ in a shear along the axis. The above observation establishes a duality between the characterizations of ellipsoids in terms of affine symmetries of their sections and in terms of linear symmetries of their supporting cones.\\

\textsl{Therefore, considering affine symmetries of supporting cones, assuming the existence of a common fixed point, is completely equivalent to considering affine symmetries of affine sections, while considering affine symmetries of supporting cones without this assumption is a significantly more difficult problem. However, considering projective symmetries of supporting cones is completely equivalent to considering projective symmetries of affine sections.}\\

In the rest of this section, we will discuss several dual pairs of such theorems. However, it should be emphasized that the map $\varphi$ is nothing but the classical projective correlation. The observation is therefore very simple, but nevertheless, it has powerful consequences, which we admittedly have not found in the literature. A different, more sophisticated duality map was introduced, e.g., in \cite{Montejano_Morales_2003}, where L. Montejano and E.~Morales-Amaya showed that several classical characterizations of ellipsoids depend upon a common concept of a projective center of symmetry.

\subsection{Flat shadow boundaries}

Let us start with the observation that \thref{thm:01}, which is the main result of this paper, is self-dual. Indeed, if for a point $p\in\mathbb A^n$ there exists a hyperplane $H\in\Graff_{n-1}(\mathbb A^n)$ such that the intersection of $H$ with any hyperplane tangent to $\partial K$ and passing through $p$ is tangent to $\partial K$, then applying $\varphi$ yields that for a hyperplane $p^\circ$ there exists a point $H^\circ$ such that the line through $H^\circ$ and any point on $\partial K^\circ$ contained in $p^\circ$ is tangent to $\partial K^\circ$, which is an equivalent property. The assumption that on every hyperplane tangent to $\partial K$ we have at least $L$ point light sources in general linear position corresponds to the assumption that every point on $\partial K^\circ$ belongs to at least $L$ flat shadow boundaries in general linear position.

\subsection{Centrally symmetric affine sections and axially symmetric supporting cones}

In \cite{Bianchi1987}, G.~Bianchi and P.M.~Gruber formulated two conjectural characterizations of ellipsoids in terms of sections and projections, which became the driving force behind many subsequent works.

\begin{conjecture}[{cf. \cite[Conjecture~1]{Bianchi1987}}]\thlabel{con:02}
Let $K\subset\mathbb R^n$, $n\geq 3$, be a strictly convex body and let $\mathcal S$ be a family of hyperplanes, which is the image of an embedding of the tangent bundle of the sphere $\mathbb S^{n-1}$, such that every hyperplane in $\mathcal S$ intersects the interior of $K$. If any such intersection is centrally symmetric and (with, possibly, a few exceptions) does not contain a possibly existing center of symmetry of $K$, then $K$ is an ellipsoid.
\end{conjecture}
%
% \begin{theorem}[{cf. \cite[Theorem~1]{Bianchi1987}}]\thlabel{thm:04}
% Let $K\subset\mathbb R^n$, $n\geq 3$, be a convex body and let $S$ be a closed convex hypersurface contained in the interior of $K$. If the section of $K$ by hyperplane $H$ is always an ellipsoid for any hyperplane $H$ tangent to $S$, then $K$ is an ellipsoid.
% \end{theorem}

\noindent As a first step towards verification, the authors proved \thref{con:02} under the stronger assumption that every section of $K$ by a hyperplane in $\mathcal S$ is an ellipsoid \cite[Theorem~1]{Bianchi1987}. Further, if $\mathcal S$ is a family of hyperplanes passing through a fixed point, then \thref{con:02} follows immediately from the result of D.G.~Larman:

\begin{theorem}[{False Center Theorem \cite[Theorem]{Larman}}]\thlabel{thm:04}
Let $K\subset\mathbb R^n$, $n\geq 3$, be a convex body, and let $p\in\mathbb R^n$ be any point of the ambient space. If all the intersections $K\cap H$ of $K$ with affine hyperplanes $H\in\Gr_{n-1}(\mathbb R^n)+p$ are centrally symmetric, then either $K$ is centrally symmetric with respect to the point $p$ or $K$ is an ellipsoid.
\end{theorem}

\noindent The aforementioned result of Larman, together with a much older result of S.P.~Olovyanishnikov \cite[Teorema and subsequent Obobshcheniye]{Olovyanishnikov}, inspired J.A.~Barker and D.G.~Larman to ask the following question:

\begin{conjecture}[{cf. \cite[Conjecture~1]{BARKER200179}}]
Let $K,S\subset\mathbb R^n$, $n\geq 3$, be convex bodies and let $S\subset\Int K$. Suppose that for every hyperplane $H$ supporting $S$, the section $K\cap H$ is centrally symmetric. Then $K$ is an ellipsoid.
\end{conjecture}

\noindent L.~Montejano and E.~Morales-Amaya verified \thref{con:02} in dimension $3$:

\begin{theorem}[{Shaken False Center Theorem \cite[Theorem~1]{S002557930000019X}}]\thlabel{thm:05}
Let $K\subset\mathbb R^3$ be an origin-symmetric convex body and let $\delta:\mathbb S^2\to\mathbb R$ be an even continuous map which is $C^1$ in a neighbourhood of $\delta^{-1}(\{0\})$. Denote $H_{\boldsymbol\xi}^\delta\colonequals\langle\boldsymbol\xi\rangle^\perp+\delta(\boldsymbol\xi)\boldsymbol\xi$, $\boldsymbol\xi\in\mathbb S^2$. If for every $\boldsymbol\xi\in\mathbb S^2$, either $\delta(\boldsymbol\xi)=0$ and $K\cap H_{\boldsymbol\xi}^\delta$ is an ellipse or $\delta(\boldsymbol\xi)\neq 0$ and $K\cap H_{\boldsymbol\xi}^\delta$ is centrally symmetric, then $K$ is an ellipsoid.
\end{theorem}

\noindent Recently, E.~Morales-Amaya proved also the following:

\begin{theorem}[{\cite[Theorem~1]{moralesamaya2023}}]\thlabel{thm:07}
Let $K\subset\mathbb R^n$, $n\geq 3$, be an origin-symmetric convex body and let $B\subset\Int K$ be a Euclidean ball that does not contain the origin. If all the intersections $K\cap H$ of $K$ with affine hyperplanes $H$ tangent to $B$ are centrally symmetric, then $K$ is an ellipsoid.
\end{theorem}

\noindent For a much more detailed account, we refer the reader to the expository paper \cite[\S 4]{Soltan2019}.\\

Now, the dual counterpart of \thref{con:02} reads:

\begin{dualtheorem}{con:02}[{cf. \cite[Conjecture~2]{Bianchi1987}}]
Let $K\subset\mathbb R^n$, $n\geq 3$, be a strictly convex body and let $S$ be a hypersurface, which is the image of an embedding of the sphere $\mathbb S^{n-1}$, such that $K$ is contained in the interior of $S$. If any supporting cone of $K$ with apex on $S$ is affinely axially symmetric and (with, possibly, a few exceptions) the axis of symmetry does not contain a possibly existing center of symmetry of $K$, then $K$ is an ellipsoid.
\end{dualtheorem}
%
% \begin{dualtheorem}{thm:04}[{cf. \cite[Theorem]{ojm/1353054527}, \cite[Theorem~2]{Bianchi1987}, \cite[Theorem~5]{Gonzalez_Garcia2022}}]
% Let $K\subset\mathbb R^n$, $n\geq 3$, be a convex body and let $S$ be a closed convex hypersurface containing $K$ in its interior. If the supporting cone of $K$ with apex at $p$ is always an ellipsoidal cone for any point $p$ on $S$, then $K$ is an ellipsoid.
% \end{dualtheorem}

\noindent Again, as a first step towards verification, the authors proved \thref{con:02}* under the stronger assumption that every supporting cone of $K$ with apex on $L$ is ellipsoidal \cite[Theorem~2]{Bianchi1987}. The dual counterpart of the False Center \thref{thm:04}, which must necessarily also be true, yields a more general variant of \cite[Theorem~4]{10.1007/s00454-024-00712-3}, and the dual counterpart of \thref{thm:07} likewise yields a more general variant of \cite[Theorem~2]{10.1007/s00454-024-00712-3}. Admittedly, many known results that do not assume the existence of a common fixed point are difficult to transfer, because their formulations in the primal setting would necessarily involve some rather special projective symmetries. This is yet another argument for the naturalness of projective transformations in this context.

\subsection{Linearly equivalent supporting cones}

In \cite{Morales_Jeronimo_Verdusco_2022}, E.~Morales-Amaya, J.~Jer\'onimo-Castro, and D.J.~Verdusco-Hern\'andez were interested in the problem of determining convex bodies $K\subset\mathbb R^n$ for which all the supporting cones with apexes in some hypersurface $S$ are linearly equivalent (cf. \cite[Problem~2]{Morales_Jeronimo_Verdusco_2022}). Since the origin is then automatically a common fixed point of all the transformations, the problem is completely equivalent by duality to the classical Banach's Conjecture \cite[Remarques~au~chapitre~XII]{banach1932theorie}. For a detailed account of this otherwise interesting problem, we refer the reader to the expository paper \cite[\S 6]{Soltan2019}. Note that the inherently topological proofs of the non-integrable Banach's Conjecture do not use the assumption that all the sections pass through a fixed point. The same question for projectively equivalent sections (and thus by duality also for projectively equivalent supporting cones), which, as we have already remarked, seems even more natural, was asked by H.~Auerbach in the Scottish Book \cite[84.~Problem: Auerbach]{scottish}.\\

However, instead of the subgroup of the general linear group, the authors eventually considered the translation group and proved the following theorem:

\begin{theorem}[{cf. \cite[Theorem~1]{Morales_Jeronimo_Verdusco_2022}}]\thlabel{thm:06}
Let $K\subset\mathbb R^n$, $n\geq 3$, be a strictly convex body and let $S$ be a hypersurface, which is the image of an embedding of the sphere $\mathbb S^{n-1}$, such that $K$ is contained in the interior of $S$. Suppose that for every $x\in S$ there exists a different $y\in S$ such that the supporting double-cones of $K$ with apexes at $x$ and $y$ differ by a translation. Then $K$ and $S$ are both centrally symmetric with respect to the same point.
\end{theorem}

\noindent Now, its dual counterpart, which must necessarily also be true, reads:

\begin{dualtheorem}{thm:06}
Let $K\subset\mathbb R^n$, $n\geq 3$, be a strictly convex body and let $\mathcal S$ be a family of hyperplanes, which is the image of an embedding of the tangent bundle of the sphere $\mathbb S^{n-1}$, such that every hyperplane in $\mathcal S$ intersects the interior of $K$. Suppose that for every $X\in\mathcal S$ there exists a different $Y\in\mathcal S$ such that the sections of $K$ by hyperplanes $X$ and $Y$ differ by a central projection with respect to the origin. Then $K$ and $S$ are both invariant under the same non-trivial projective involution fixing the origin.
\end{dualtheorem}

\noindent Since every projective involution between two embedded hyperplanes is a restriction of a central projection, it is natural to remove the dependence on the choice of the origin and ask the following, slightly more general question:

\begin{conjecture}\thlabel{con:04}
Let $K\subset\mathbb R^n$, $n\geq 3$, be a strictly convex body and let $\mathcal S$ be a family of hyperplanes, which is the image of an embedding of the tangent bundle of the sphere $\mathbb S^{n-1}$, such that every hyperplane in $\mathcal S$ intersects the interior of $K$. Suppose that for every $X\in\mathcal S$ there exists a different $Y\in\mathcal S$ such that the sections of $K$ by hyperplanes $X$ and $Y$ differ by a non-trivial projective involution. Then $K$ and $\mathcal S$ are both invariant under the same non-trivial projective involution.
\end{conjecture}

\noindent Further, since every homothety between two parallel embedded hyperplanes is a restriction of a projective involution, \thref{con:04} may be viewed as a projective counterpart of the following question:

\begin{conjecture}[{cf. \cite[Conjecture~1]{Morales_Jeronimo_Verdusco_2022}}]\thlabel{con:05}
Let $K,S\subset\mathbb R^n$, $n\geq 3$, be convex bodies and let $S\subset\Int K$. Suppose that for every pair of parallel hyperplanes $H_1,H_2$ supporting $S$, the sections $K\cap H_1,K\cap H_2$ are inversely homothetic with respect to the point lying on the segment connecting the contact points of $K$ with $H_1,H_2$. Then $K,L$ are centrally symmetric with respect to the same point.
\end{conjecture}

\subsection{Flat intersections of supporting cones}

Finally, in a recent paper \cite{moralesamaya2025characterizationsellipsoidsmeansstrong}, E.~Morales-Amaya considered a different, interesting problem, which is another piece completing our puzzle, expanding it in a new direction. Namely, he posed the following conjecture:

\begin{conjecture}[{cf. \cite[Conjecture~1]{moralesamaya2025characterizationsellipsoidsmeansstrong}}]\thlabel{con:03}
Let $K\subset\mathbb R^n$, $n\geq 3$, be a convex body and let $S$ be a hypersurface, which is the image of an embedding of the sphere $\mathbb S^{n-1}$, such that $K$ is contained in the interior of $S$, and $S$ is star-convex with respect to the origin. Suppose that for every $x\in S$ there exists a different $y\in S$ such that the line $xy$ passes through $o$ and the intersection of the boundaries of supporting cones of $K$ with apexes at $x$ and $y$ is flat. Then $K$ is an ellipsoid.
\end{conjecture}

\noindent Now, its dual counterpart, which must necessarily be equivalent, reads:

\begin{dualtheorem}{con:03}
Let $K\subset\mathbb R^n$, $n\geq 3$, be a convex body and let $\mathcal S$ be a family of hyperplanes, which is the image of an embedding of the tangent bundle of the sphere $\mathbb S^{n-1}$, such that every hyperplane in $\mathcal S$ intersects the interior of $K$. Suppose that for every $X\in\mathcal S$ there exists a different $Y\in\mathcal S$ such that $X$ and $Y$ are parallel and the sections $K\cap X$ and $K\cap Y$ are directly homothetic. Then $K$ is an ellipsoid.
\end{dualtheorem}

\noindent Indeed, since $(\ell_1\cap\ell_2)^\circ=\Span\{\ell_1^\circ,\ell_2^\circ\}$, two lines $\ell_1,\ell_2$ passing through points $x,y$ respectively, intersect if and only if two $2$-codimensional flats $\ell_1^\circ,\ell_2^\circ$ contained in parallel hyperplanes $x^\circ,y^\circ$ respectively, span a hyperplane, i.e., are parallel. Further, the point of intersection $\ell_1\cap\ell_2$ lies on a hyperplane $H$ if and only if the hyperplane spanned by $\ell_1^\circ,\ell_2^\circ$ passes through a point $H^\circ$. It follows that the support functions of $K\cap x^\circ$ and $K\cap y^\circ$ are proportional, whence the sections are homothetic with respect to $H^\circ$.\\

Note that for every $2$-codimensional flat supporting $K\cap x^\circ$ and contained in $x^\circ$, there are precisely two parallel $2$-dimensional flats supporting $K\cap y^\circ$ and contained in $y^\circ$. The sign of homothety depends on the choice of one of them, which corresponds to the choice of which parts of the supporting double cones we intersect in the primal setting. In our case, because we consider the hyperplane $H$ intersecting the interior of $K$ in the primal setting, we must have the center of homothety $H^\circ$ contained in the exterior of $S^\circ$ in the dual setting, which corresponds to the positive-scale one.\\

Under the additional assumption that the intersection of the boundaries of supporting cones of $K$ with apexes at $x$ and $y$ is contained in some convex hypersurface, \thref{con:03} was confirmed by the author \cite[Theorem~3]{moralesamaya2025characterizationsellipsoidsmeansstrong}. Note the striking similarity between \thref{con:03}* and \thref{con:05}. The only difference is the sign of the scale factor of the homothety. The author asked in \cite[Problem~1]{moralesamaya2025characterizationsellipsoidsmeansstrong} whether \thref{con:03} holds under the additional assumption that $K$ and $S$ are origin-symmetric. In this case, in the equivalent \thref{con:03}*, we have that the sections $K\cap X$ and $K\cap Y$ are both directly and inversely homothetic, which means that they are both centrally symmetric. Hence the problem is reduced to \thref{con:02}, which was already confirmed in dimension $3$ (cf. \thref{thm:05}).\\

Note that if we additionally assume that the hyperplane containing the intersection of the boundaries of supporting cones of $K$ with apexes at $x$ and $y$ passes through the origin, then the corresponding sections $K\cap x^\circ$ and $K\cap y^\circ$ differ by a translation. In this disguise, the conjecture is already known to hold in dimension $3$ \cite[Theorem~1]{S0025579310000379}. Under assumptions that $S$ is the unit sphere and the center of symmetry of each section is at the contact point, the conjecture was independently confirmed in arbitrary dimension in \cite[Theorem~2]{Gonzalez_Garcia2022}.

%%%%%%%%%%%%%%%%%%%%%%%%%%%%%%%%%%%%%%%%%%%%%%%%%%%%%%%%%%%%%%%%%%%%%%
\section{Definitions and basic concepts}
%%%%%%%%%%%%%%%%%%%%%%%%%%%%%%%%%%%%%%%%%%%%%%%%%%%%%%%%%%%%%%%%%%%%%%

In this section, we recall several notions and classical results from differential geometry that we will use in our proofs. For background on Riemannian differential geometry, we refer the interested reader to \cite{do2016differential}, and for affine differential geometry to \cite{nomizu1994affine,Buchin}.\\

Let $M\subset\mathbb R^n$ be a regular orientable hypersurface in which a differentiable field of unit normal vectors $\boldsymbol N$ has been chosen.

\begin{definition}[{\cite[\S 3-2, Definition~1]{do2016differential}}]
Let $M\subset\mathbb R^n$ be a hypersurface with an orientation $\boldsymbol N$. The map $\boldsymbol N:M\to\mathbb S^{n-1}$ is called the \emph{Gauss map} of $M$.
\end{definition}

\noindent It is straightforward to verify that the Gauss map is differentiable. The differential $\mathrm d\boldsymbol N_p$ of $\boldsymbol N$ at $p\in M$ is a linear map from $T_pM$ to $T_{\boldsymbol N(p)}\mathbb S^{n-1}$. Since $T_pM$ and $T_{\boldsymbol N(p)}\mathbb S^{n-1}$ are parallel planes, $\mathrm d\boldsymbol N_p$ can be looked upon as a linear map on $T_pM$.

\begin{definition}[{\cite[\S 3-2, Definition~2]{do2016differential}}]
The quadratic form $\mathrm{II}_p$, defined in $T_pM$ by $\mathrm{II}_p(\boldsymbol v)=-\langle\mathrm d\boldsymbol N_p(\boldsymbol v),\boldsymbol v\rangle$ is called the \emph{second fundamental form} of $M$ at $p$.
\end{definition}

The second fundamental form has rank $n-1$, and can be treated as a non-degenerate metric on $M$. This is the basic assumption on which W.~Blaschke developed the affine differential geometry of hypersurfaces.

\begin{definition}[{\cite[II.3]{nomizu1994affine}}]
It is well known that there exists a canonical choice of a transversal vector field $\xi$ called the \emph{affine normal field} or \emph{Blaschke normal field} \cite[Definition~II.3.1]{nomizu1994affine}. The affine normal vector field $\xi$ gives rise to the induced connection $\nabla$, the affine fundamental form $h$, which is traditionally called the \emph{affine metric}, and the affine shape operator $S$ determined by the formulas
\begin{gather*}
D_XY=\nabla_XY+h(X,Y)\xi,\\
D_X\xi=-SX.
\end{gather*}
We shall call $(\nabla,h,S)$ the \emph{Blaschke structure} on the hypersurface $M$ \cite[Definition~II.3.2]{nomizu1994affine}. From the Codazzi equation for $h$, we see that the cubic form
\begin{equation}\label{eq:15}C(X,Y,Z)\colonequals(\nabla_Xh)(Y,Z)\end{equation}
is symmetric in $X$, $Y$ and $Z$ \cite[II.4]{nomizu1994affine}.
\end{definition}

\noindent The following is an important classical theorem due to H.~Maschke (for analytic surfaces), G.A.~Pick (for surfaces), and L.~Berwald (for hypersurfaces):

\begin{theorem}[{\cite[Theorem~II.4.5]{nomizu1994affine}}]\thlabel{thm:03}
Let $f:M\to\mathbb R^n$, $n\geq 3$, be a non-degenerate hypersurface with Blaschke structure. If the cubic form \eqref{eq:15} vanishes identically, then $f(M)$ is hyperquadric in $\mathbb R^n$.
\end{theorem}

\noindent Actually, the following generalization is known:

\begin{theorem}[{\cite[Corollary~5.4]{Zawalski2025}, see the revised version \cite[Corollary~5.7]{Zawalski_2024}}]\thlabel{thm:08}
Let $M\subset\mathbb R^n$, $n\geq 3$, be a convex hypersurface of class $C^{2,1}$ such that for almost every $p\in M$ there is a quadratic hypersurface having third-order contact with $M$ at $p$. Then $M$ is itself a quadratic hypersurface.
\end{theorem}

\noindent Importantly, it requires only that the hypersurface be of class $C^{2,1}$, whereas the theorem of Maschke-Pick-Berwald, as formulated in \cite[Theorem~II.4.5]{nomizu1994affine}, requires the hypersurface to be smooth, and its proof indeed relies on higher-order tensors.

%%%%%%%%%%%%%%%%%%%%%%%%%%%%%%%%%%%%%%%%%%%%%%%%%%%%%%%%%%%%%%%%%%%%%%
\section{Infinitesimal properties of flat shadow boundaries}
%%%%%%%%%%%%%%%%%%%%%%%%%%%%%%%%%%%%%%%%%%%%%%%%%%%%%%%%%%%%%%%%%%%%%%

Denote the submanifold $\partial K$ by $M^{n-1}$. Let $p\in M^{n-1}$ be any point of $M^{n-1}$ with positive definite second fundamental form. After applying a suitable affine map we may assume that $p=\mathbf 0_{\mathbb R^n}$ and $T_pM^{n-1}=\mathbf 0_{\mathbb R^n}+\langle\hat{\boldsymbol e}_n\rangle^\perp$, where $\hat{\boldsymbol e}_{n}$ stands for the $n$\textsuperscript{th} standard unit vector (fig.~\ref{fig:01}). In this coordinate system, we may represent the neighborhood of $p$ in $M^{n-1}$ as a graph of some function $f:T_pM^{n-1}\supset U\to\mathbb R$ of class $C^3$, which must be of the form
$$f(\boldsymbol x)=O(\|\boldsymbol x\|)^2$$
in Big $O$ notation. Since we assumed that the second fundamental form of $M^{n-1}$ is positive definite at $p$, after applying a suitable linear change of coordinates in the domain, we may further assume that
$$f(\boldsymbol x)=\frac{1}{2}\langle\boldsymbol x,\boldsymbol x\rangle+O(\|\boldsymbol x\|)^3.$$
The above will be called a \emph{canonical parametrization} of $M^{n-1}$ at $p$.

\begin{remark}
Even though the second-order jet of $f$ is already fixed, we can still apply a linear transformation
$$\begin{pmatrix}\mathbf I_{n-1}&\boldsymbol\eta\\\mathbf 0_{n-1}^\top&1\end{pmatrix}\in\SL(\mathbb R^n),$$
which effectively adds a cubic form
\begin{equation}\label{eq:11}-\frac{1}{2}\langle\boldsymbol x,\boldsymbol x\rangle\langle\boldsymbol x,\boldsymbol\eta\rangle \end{equation}
to the third-order jet of $f$. However, we will not do this now, but rather choose a suitable transformation later. After that, the canonical parametrization will be unique up to an orthogonal change of coordinates in the domain.
\end{remark}

\begin{figure}
\includegraphics{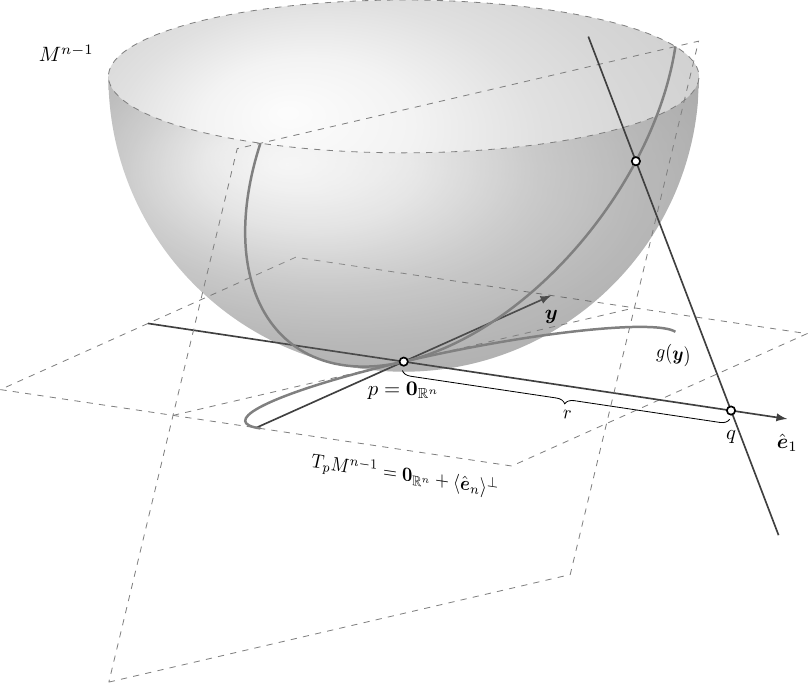}
\caption{Notations used in the proof}
\label{fig:01}
\end{figure}

\subsection{Local parametrization of a shadow boundary}

Let $q$ be any point light source that lies on $T_pM^{n-1}=\langle\hat{\boldsymbol e}_n\rangle^\perp$. After applying a suitable orthogonal change of coordinates, we may further assume that $q=r\hat{\boldsymbol e}_1$ for some $r>0$.

\begin{claim}\thlabel{lem:01}
The shadow boundary is transversal to the supporting line $pq=\langle\hat{\boldsymbol e}_1\rangle$. Indeed, otherwise the hyperplane spanned by the shadow boundary would contain $q$, whence almost every line passing through $q$ and intersecting the shadow boundary would intersect the interior of $K$, a contradiction.
\end{claim}

It follows that also the projection of the shadow boundary onto the tangent hyperplane $T_pM^{n-1}=\langle\hat{\boldsymbol e}_n\rangle^\perp$ is transversal to the supporting line $\langle\hat{\boldsymbol e}_1\rangle$. Hence, we may represent the neighbourhood of $p$ in this projection as a graph of some function $g:\langle\hat{\boldsymbol e}_1,\hat{\boldsymbol e}_n\rangle^\perp\supset V\to\mathbb R$ of class $C^3$. In this parametrization, points on the shadow boundary are of the form $(g(\boldsymbol y),\boldsymbol y,f(g(\boldsymbol y),\boldsymbol y))$ for $\boldsymbol y\in V$. A (not necessarily unit) normal vector to $M^{n-1}$ at such a point is of the form $((\nabla f)(g(\boldsymbol y),\boldsymbol y),-1)$. Thus, the definition of the shadow boundary reads
$$((\nabla_1f)(g(\boldsymbol y),\boldsymbol y),(\nabla_{-1}f)(g(\boldsymbol y),\boldsymbol y),-1)\perp(g(\boldsymbol y)-r,\boldsymbol y,f(g(\boldsymbol y),\boldsymbol y)),$$
which is equivalent to
\begin{equation}\label{eq:01}f_1(g(\boldsymbol y),\boldsymbol y)(g(\boldsymbol y)-r)+(Df)(g(\boldsymbol y),\boldsymbol y)[\boldsymbol y]-f(g(\boldsymbol y),\boldsymbol y)=0,\end{equation}
where $D$ is the directional derivative operator on $\langle\hat{\boldsymbol e}_1,\hat{\boldsymbol e}_n\rangle^\perp$. From now on, for brevity, we will omit the arguments of functions whenever they are clear from the context. Differentiating \eqref{eq:01} with respect to $\boldsymbol y$ yields
$$f_{11}Dg[\boldsymbol a](g-r)+Df_1[\boldsymbol a](g-r)+\cancel{f_1Dg[\boldsymbol a]}+Df_1[\boldsymbol y]Dg[\boldsymbol a]+D^2f[\boldsymbol y,\boldsymbol a]+\cancel{Df[\boldsymbol a]}-\cancel{f_1Dg[\boldsymbol a]}-\cancel{Df[\boldsymbol a]}=0.$$
Hence
$$Dg[\boldsymbol a]=(-Df_1[\boldsymbol a] (g-r) -D^2f[\boldsymbol y,\boldsymbol a]) (f_{11} (g-r) +Df_1[\boldsymbol y])^{-1}.$$
Note that the denominator does not vanish in some neighbourhood of $\boldsymbol y=0$. Differentiating the above equation with respect to $\boldsymbol y$ yields
\begin{align*}
&D^2g[\boldsymbol a,\boldsymbol b]=\big(-(Df_{11}[\boldsymbol a] Dg[\boldsymbol b] +D^2f_1[\boldsymbol a,\boldsymbol b]) (g-r) -Df_1[\boldsymbol a] Dg[\boldsymbol b] -D^2f_1[\boldsymbol y,\boldsymbol a] Dg[\boldsymbol b] -D^3f[\boldsymbol y,\boldsymbol a,\boldsymbol b]\\
&\quad -D^2f[\boldsymbol b,\boldsymbol a]\big) (f_{11} (g-r) +Df_1[\boldsymbol y])^{-1} -(-Df_1[\boldsymbol a] (g-r) -D^2f[\boldsymbol y,\boldsymbol a]) \big((f_{111} Dg[\boldsymbol b] +Df_{11}[\boldsymbol b]) (g-r)\\
&\quad +f_{11} Dg[\boldsymbol b] +Df_{11}[\boldsymbol y] Dg[\boldsymbol b] +D^2f_1[\boldsymbol y,\boldsymbol b] +Df_1[\boldsymbol b]\big) (f_{11} (g-r) +Df_1[\boldsymbol y])^{-2}.
\end{align*}
Again, differentiating the above equation with respect to $\boldsymbol y$ yields
\begin{align*}
&D^3g[\boldsymbol a,\boldsymbol b,\boldsymbol c]=\big(-((Df_{111}[\boldsymbol a] Dg[\boldsymbol c] +D^2f_{11}[\boldsymbol a,\boldsymbol c]) Dg[\boldsymbol b] +Df_{11}[\boldsymbol a] D^2g[\boldsymbol b,\boldsymbol c] +D^2f_{11}[\boldsymbol a,\boldsymbol b] Dg[\boldsymbol c]\\
&\quad +D^3f_1[\boldsymbol a,\boldsymbol b,\boldsymbol c]) (g-r) -(Df_{11}[\boldsymbol a] Dg[\boldsymbol b] +D^2f_1[\boldsymbol a,\boldsymbol b]) Dg[\boldsymbol c] -(Df_{11}[\boldsymbol a] Dg[\boldsymbol c] +D^2f_1[\boldsymbol a,\boldsymbol c]) Dg[\boldsymbol b]\\
&\quad -Df_1[\boldsymbol a] D^2g[\boldsymbol b,\boldsymbol c] -(D^2f_{11}[\boldsymbol y,\boldsymbol a] Dg[\boldsymbol c] +D^3f_1[\boldsymbol y,\boldsymbol a,\boldsymbol c] +D^2f_1[\boldsymbol c,\boldsymbol a]) Dg[\boldsymbol b] -D^2f_1[\boldsymbol y,\boldsymbol a] D^2g[\boldsymbol b,\boldsymbol c]\\
&\quad -D^3f_1[\boldsymbol y,\boldsymbol a,\boldsymbol b] Dg[\boldsymbol c] -D^4f[\boldsymbol y,\boldsymbol a,\boldsymbol b,\boldsymbol c] -D^3f[\boldsymbol c,\boldsymbol a,\boldsymbol b] -D^2f_1[\boldsymbol b,\boldsymbol a] Dg[\boldsymbol c] -D^3f[\boldsymbol b,\boldsymbol a,\boldsymbol c]\big) (f_{11} (g-r)\\
&\quad +Df_1[\boldsymbol y])^{-1} -(-(Df_{11}[\boldsymbol a] Dg[\boldsymbol b] +D^2f_1[\boldsymbol a,\boldsymbol b]) (g-r) -Df_1[\boldsymbol a] Dg[\boldsymbol b] -D^2f_1[\boldsymbol y,\boldsymbol a] Dg[\boldsymbol b] -D^3f[\boldsymbol y,\boldsymbol a,\boldsymbol b]\\
&\quad -D^2f[\boldsymbol b,\boldsymbol a]) ((f_{111} Dg[\boldsymbol c] +Df_{11}[\boldsymbol c]) (g-r) +f_{11} Dg[\boldsymbol c] +Df_{11}[\boldsymbol y] Dg[\boldsymbol c] +D^2f_1[\boldsymbol y,\boldsymbol c] +Df_1[\boldsymbol c])\\
&\quad (f_{11} (g-r) +Df_1[\boldsymbol y])^{-2} -(-(Df_{11}[\boldsymbol a] Dg[\boldsymbol c] +D^2f_1[\boldsymbol a,\boldsymbol c]) (g-r) -Df_1[\boldsymbol a] Dg[\boldsymbol c] -D^2f_1[\boldsymbol y,\boldsymbol a] Dg[\boldsymbol c]\\
&\quad -D^3f[\boldsymbol y,\boldsymbol a,\boldsymbol c] -D^2f[\boldsymbol c,\boldsymbol a]) ((f_{111} Dg[\boldsymbol b] +Df_{11}[\boldsymbol b]) (g-r) +f_{11} Dg[\boldsymbol b] +Df_{11}[\boldsymbol y] Dg[\boldsymbol b] +D^2f_1[\boldsymbol y,\boldsymbol b]\\
&\quad +Df_1[\boldsymbol b]) (f_{11} (g-r) +Df_1[\boldsymbol y])^{-2} -(-Df_1[\boldsymbol a] (g-r) -D^2f[\boldsymbol y,\boldsymbol a]) \big(\big(((f_{1111} Dg[\boldsymbol c] +Df_{111}[\boldsymbol c]) Dg[\boldsymbol b]\\
&\quad +f_{111} D^2g[\boldsymbol b,\boldsymbol c] +Df_{111}[\boldsymbol b] Dg[\boldsymbol c] +D^2f_{11}[\boldsymbol b,\boldsymbol c]) (g-r) +(f_{111} Dg[\boldsymbol b] +Df_{11}[\boldsymbol b]) Dg[\boldsymbol c] +(f_{111} Dg[\boldsymbol c]\\
&\quad +Df_{11}[\boldsymbol c]) Dg[\boldsymbol b] +f_{11} D^2g[\boldsymbol b,\boldsymbol c] +(Df_{111}[\boldsymbol y] Dg[\boldsymbol c] +D^2f_{11}[\boldsymbol y,\boldsymbol c] +Df_{11}[\boldsymbol c]) Dg[\boldsymbol b] +Df_{11}[\boldsymbol y] D^2g[\boldsymbol b,\boldsymbol c]\\
&\quad +D^2f_{11}[\boldsymbol y,\boldsymbol b] Dg[\boldsymbol c] +D^3f_1[\boldsymbol y,\boldsymbol b,\boldsymbol c] +D^2f_1[\boldsymbol c,\boldsymbol b] +Df_{11}[\boldsymbol b] Dg[\boldsymbol c] +D^2f_1[\boldsymbol b,\boldsymbol c]\big) (f_{11} (g-r) +Df_1[\boldsymbol y])^{-2}\\
&\quad -2 ((f_{111} Dg[\boldsymbol b] +Df_{11}[\boldsymbol b]) (g-r) +f_{11} Dg[\boldsymbol b] +Df_{11}[\boldsymbol y] Dg[\boldsymbol b] +D^2f_1[\boldsymbol y,\boldsymbol b] +Df_1[\boldsymbol b]) ((f_{111} Dg[\boldsymbol c]\\
&\quad +Df_{11}[\boldsymbol c]) (g-r) +f_{11} Dg[\boldsymbol c] +Df_{11}[\boldsymbol y] Dg[\boldsymbol c] +D^2f_1[\boldsymbol y,\boldsymbol c] +Df_1[\boldsymbol c]) (f_{11} (g-r) +Df_1[\boldsymbol y])^{-3}\big).
\end{align*}
Finally, recall that
\begin{equation}\label{eq:05}g(\mathbf 0)=0,\ f_1(\mathbf 0)=0,\ Df(\mathbf 0)[\boldsymbol a]=0,\ f_{11}(\mathbf 0)=1,\ Df_1(\mathbf 0)[\boldsymbol a]=0,\ D^2f(\mathbf 0)[\boldsymbol a,\boldsymbol b]=\langle\boldsymbol a,\boldsymbol b\rangle.\end{equation}
Evaluating the above equations at $\boldsymbol y=0$ gives us
\begin{equation}\label{eq:02}\begin{aligned}
Dg(\mathbf 0)[\boldsymbol a]={}&0,\\
D^2g(\mathbf 0)[\boldsymbol a,\boldsymbol b]={}&-D^2f_1[\boldsymbol a,\boldsymbol b] +r^{-1} \langle\boldsymbol a,\boldsymbol b\rangle,\\
D^3g(\mathbf 0)[\boldsymbol a,\boldsymbol b,\boldsymbol c]={}&-D^3f_1[\boldsymbol a,\boldsymbol b,\boldsymbol c] +2 r^{-1} D^3f[\boldsymbol a,\boldsymbol b,\boldsymbol c] +(D^2f_1[\boldsymbol a,\boldsymbol b] -r^{-1} \langle\boldsymbol a,\boldsymbol b\rangle) Df_{11}[\boldsymbol c]\\
&\quad +(D^2f_1[\boldsymbol b,\boldsymbol c] -r^{-1} \langle\boldsymbol b,\boldsymbol c\rangle) Df_{11}[\boldsymbol a] +(D^2f_1[\boldsymbol c,\boldsymbol a] -r^{-1} \langle\boldsymbol c,\boldsymbol a\rangle) Df_{11}[\boldsymbol b].
\end{aligned}\end{equation}
In particular, \emph{a posteriori} we may conclude that the shadow boundary is not only transversal to $\langle\hat{\boldsymbol e}_1\rangle$, but even orthogonal.

\begin{remark}
The same argument yields further
\begin{align*}
&D^4g(\mathbf 0)[\boldsymbol a,\boldsymbol b,\boldsymbol c,\boldsymbol d]=-D^4f_1[\boldsymbol a,\boldsymbol b,\boldsymbol c,\boldsymbol d] +3 r^{-1} D^4f[\boldsymbol a,\boldsymbol b,\boldsymbol c,\boldsymbol d] +\sum D^3f_1[\boldsymbol a,\boldsymbol b,\boldsymbol c] Df_{11}[\boldsymbol d]\\
&\quad +\sum D^2f_{11}[\boldsymbol a,\boldsymbol b] (D^2f_1[\boldsymbol c,\boldsymbol d] -r^{-1} \langle\boldsymbol c,\boldsymbol d\rangle) -2 r^{-1} \sum D^3f[\boldsymbol a,\boldsymbol b,\boldsymbol c] Df_{11}[\boldsymbol d]\\
&\quad -(f_{111} +3 r^{-1}) \sum D^2f_1[\boldsymbol a,\boldsymbol b] D^2f_1[\boldsymbol c,\boldsymbol d] -2 \sum (D^2f_1[\boldsymbol a,\boldsymbol b] -r^{-1} \langle\boldsymbol a,\boldsymbol b\rangle) Df_{11}[\boldsymbol c] Df_{11}[\boldsymbol d]\\
&\quad +r^{-1} (f_{111} +r^{-1}) \sum (D^2f_1[\boldsymbol a,\boldsymbol b] -r^{-1} \langle\boldsymbol a,\boldsymbol b\rangle) \langle\boldsymbol c,\boldsymbol d\rangle +2 r^{-3} \sum \langle\boldsymbol a,\boldsymbol b\rangle \langle\boldsymbol c,\boldsymbol d\rangle,
\end{align*}
where each sum runs over all the permutations of $\boldsymbol a,\boldsymbol b,\boldsymbol c,\boldsymbol d$, and is normalized by the order of the stabilizer subgroup of its corresponding term. For brevity, we will skip the calculations.
\end{remark}

\subsection{Local parametrization of a section}

Let $H\in\Gr_{n-1}(\mathbb R^n)$ be any hyperplane. After applying a suitable orthogonal change of coordinates, we may further assume that $H=(-\hat{\boldsymbol e}_1+m\hat{\boldsymbol e}_n)^\perp$ for some $m>0$. The projection of the section onto the tangent hyperplane $T_pM^{n-1}=\langle\hat{\boldsymbol e}_n\rangle^\perp$ is orthogonal to the supporting line $\langle\hat{\boldsymbol e}_1\rangle$. Hence, we may represent the neighbourhood of $p$ in this projection as a graph of some function $g:\langle\hat{\boldsymbol e}_1,\hat{\boldsymbol e}_n\rangle^\perp\supset V\to\mathbb R$ of class $C^3$. In this parametrization, points on the section are of the form $(g(\boldsymbol y),\boldsymbol y,f(g(\boldsymbol y),\boldsymbol y))$ for $\boldsymbol y\in V$. Thus, the definition of the section reads
$$(-1,\mathbf 0,m)\perp(g(\boldsymbol y),\boldsymbol y,f(g(\boldsymbol y),\boldsymbol y)),$$
which is equivalent to
\begin{equation}\label{eq:03}-g(\boldsymbol y)+mf(g(\boldsymbol y),\boldsymbol y)=0.\end{equation}
From now on, for brevity, we will omit the arguments of functions whenever they are clear from the context. Differentiating \eqref{eq:03} with respect to $\boldsymbol y$ yields
$$-Dg[\boldsymbol a]+mf_1Dg[\boldsymbol a]+mDf[\boldsymbol a]=0.$$
Hence
$$Dg[\boldsymbol a]=m Df[\boldsymbol a] (1 -m f_1)^{-1}.$$
Note that the denominator does not vanish in some neighbourhood of $\boldsymbol y=0$. Differentiating the above equation with respect to $\boldsymbol y$ yields
$$D^2g[\boldsymbol a,\boldsymbol b]=m (Df_1[\boldsymbol a] Dg[\boldsymbol b] +D^2f[\boldsymbol a,\boldsymbol b]) (1 -m f_1)^{-1} +m^2 Df[\boldsymbol a] (f_{11} Dg[\boldsymbol b] +Df_1[\boldsymbol b]) (1 -m f_1)^{-2}.$$
Again, differentiating the above equation with respect to $\boldsymbol y$ yields
\begin{align*}
&D^3g[\boldsymbol a,\boldsymbol b,\boldsymbol c]=m ((Df_{11}[\boldsymbol a] Dg[\boldsymbol c] +D^2f_1[\boldsymbol a,\boldsymbol c]) Dg[\boldsymbol b] +Df_1[\boldsymbol a] D^2g[\boldsymbol b,\boldsymbol c] +D^2f_1[\boldsymbol a,\boldsymbol b] Dg[\boldsymbol c]\\
&\quad +D^3f[\boldsymbol a,\boldsymbol b,\boldsymbol c]) (1 -m f_1)^{-1} +m^2 (Df_1[\boldsymbol a] Dg[\boldsymbol b] +D^2f[\boldsymbol a,\boldsymbol b]) (f_{11} Dg[\boldsymbol c] +Df_1[\boldsymbol c]) (1 -m f_1)^{-2}\\
&\quad +m^2 (Df_1[\boldsymbol a] Dg[\boldsymbol c] +D^2f[\boldsymbol a,\boldsymbol c]) (f_{11} Dg[\boldsymbol b] +Df_1[\boldsymbol b]) (1 -m f_1)^{-2} +m Df[\boldsymbol a] \big(m ((f_{111} Dg[\boldsymbol c]\\
&\quad +Df_{11}[\boldsymbol c]) Dg[\boldsymbol b] +f_{11} D^2g[\boldsymbol b,\boldsymbol c] +Df_{11}[\boldsymbol b] Dg[\boldsymbol c] +D^2f_1[\boldsymbol b,\boldsymbol c]) (1 -m f_1)^{-2} +2 m^2 (f_{11} Dg[\boldsymbol b] +Df_1[\boldsymbol b])\\
&\quad (f_{11} Dg[\boldsymbol c] +Df_1[\boldsymbol c]) (1 -m f_1)^{-3}\big).
\end{align*}
Finally, recall \eqref{eq:05}. Evaluating the above equations at $\boldsymbol y=0$ gives us
\begin{equation}\label{eq:04}\begin{aligned}
Dg(\mathbf 0)[\boldsymbol a]={}&0,\\
D^2g(\mathbf 0)[\boldsymbol a,\boldsymbol b]={}&m \langle\boldsymbol a,\boldsymbol b\rangle,\\
D^3g(\mathbf 0)[\boldsymbol a,\boldsymbol b,\boldsymbol c]={}&m D^3f[\boldsymbol a,\boldsymbol b,\boldsymbol c].
\end{aligned}\end{equation}

\begin{remark}
The same argument yields further
$$D^4g(\mathbf 0)[\boldsymbol a,\boldsymbol b,\boldsymbol c,\boldsymbol d]=m D^4f[\boldsymbol a,\boldsymbol b,\boldsymbol c,\boldsymbol d] +m^2 \sum (D^2f_1[\boldsymbol a,\boldsymbol b] +m \langle\boldsymbol a,\boldsymbol b\rangle) \langle\boldsymbol c,\boldsymbol d\rangle.$$
Again, we will skip the calculations.
\end{remark}

\subsection{Constraints resulting from flatness of the shadow boundary}

So far, we have considered a completely general situation. From now henceforth, assume that the shadow boundary is contained in a hyperplane $H\in\Gr_{n-1}(\mathbb R^n)$. Thus, it is equal to the section $K\cap H$. Further, because it is orthogonal to$\langle\hat{\boldsymbol e}_1\rangle$, the hyperplane $H$ must contain $\langle\hat{\boldsymbol e}_1,\hat{\boldsymbol e}_n\rangle^\perp$. Hence, we may simply equate \eqref{eq:02} and \eqref{eq:04}. This gives us
\begin{align}
\label{eq:06}D^2f_1[\boldsymbol a,\boldsymbol b]={}&(-m +r^{-1}) \langle\boldsymbol a,\boldsymbol b\rangle,\\
\label{eq:07}D^3f_1[\boldsymbol a,\boldsymbol b,\boldsymbol c]={}&(-m +2 r^{-1}) D^3f[\boldsymbol a,\boldsymbol b,\boldsymbol c] -m Df_{11}[\boldsymbol a] \langle\boldsymbol b,\boldsymbol c\rangle -m Df_{11}[\boldsymbol b] \langle\boldsymbol c,\boldsymbol a\rangle -m Df_{11}[\boldsymbol c] \langle\boldsymbol a,\boldsymbol b\rangle.
\end{align}
The same argument yields further
\begin{align}
\label{eq:08}D^4f_1[\boldsymbol a,\boldsymbol b,\boldsymbol c,\boldsymbol d]={}&(-m +3 r^{-1}) D^4f[\boldsymbol a,\boldsymbol b,\boldsymbol c,\boldsymbol d] -m \sum D^2f_{11}[\boldsymbol a,\boldsymbol b] \langle\boldsymbol c,\boldsymbol d\rangle -m \sum D^3f[\boldsymbol a,\boldsymbol b,\boldsymbol c] Df_{11}[\boldsymbol d]\\
\nonumber&\quad -((-m +r^{-1}) (-4 m +r^{-1}) r^{-1} +f_{111} (m^2 -m r^{-1} +r^{-2})) \sum \langle\boldsymbol a,\boldsymbol b\rangle \langle\boldsymbol c,\boldsymbol d\rangle.
\end{align}
Continuing, we will obtain similar expressions for $D^5f_1,D^6f_1,\ldots$, provided that the function $f$ is sufficiently smooth.\\

\begin{remark}
So far, we have considered only a single point and a single flat shadow boundary that passes through it. Assuming that this is the case for every point in some open neighbourhood, we can differentiate our equations with respect to the base point, obtaining potentially new constraints. However, the main difficulty is that the coordinate system and the parameters $m,r^{-1}$ depend on the base point. To overcome this issue, we can use the first two equations \eqref{eq:06}, \eqref{eq:07} to eliminate $m,r^{-1}$ from the third equation \eqref{eq:08}, but in this way we will obtain an equation containing derivatives of up to the sixth order. Besides, by assuming that flat shadow boundaries are created by multiple light sources from various directions, and thus that equations \eqref{eq:06}--\eqref{eq:08} are satisfied in suitably rotated coordinate systems, we will break the asymmetry and obtain that the same equations are satisfied in every direction. In other words, that every shadow boundary passing through the origin is flat, up to the terms of certain order. Hence, differentiation with respect to the base point most probably will not give us any new information beyond what we can deduce from pointwise information using standard affine differential geometry tools.
\end{remark}

%%%%%%%%%%%%%%%%%%%%%%%%%%%%%%%%%%%%%%%%%%%%%%%%%%%%%%%%%%%%%%%%%%%%%%
\section{Proofs of the main theorems}\label{sec:02}
%%%%%%%%%%%%%%%%%%%%%%%%%%%%%%%%%%%%%%%%%%%%%%%%%%%%%%%%%%%%%%%%%%%%%%

Again, let $p\in M^{n-1}$ be any point of $M^{n-1}$ with positive definite second fundamental form, and let $f$ be the canonical parametrization of $M^{n-1}$ at $p$. Suppose that there are $L$ point light sources $u_1,u_2,\ldots,u_L$ on the tangent hyperplane $T_pM^{n-1}$ in a general position with respect to $p$, that create flat shadow boundaries on $M^{n-1}$, and denote $\boldsymbol u_i\colonequals u_i-p$. Then \eqref{eq:06} reads
$$D^3f[\boldsymbol u_i,\boldsymbol a,\boldsymbol b]=\lambda_i\langle\boldsymbol a,\boldsymbol b\rangle$$
for every $\boldsymbol a,\boldsymbol b\in\langle\boldsymbol u_i\rangle^\perp$ and $i=1,2,\ldots,L$. In other words, the quadratic form
$$D^3f[\boldsymbol u_i,\boldsymbol x,\boldsymbol x]-\lambda_i\langle\boldsymbol x,\boldsymbol x\rangle$$
vanishes on $\langle\boldsymbol u_i\rangle^\perp$. Thus, it is reducible and admits a linear factor $\langle\boldsymbol x,\boldsymbol u_i\rangle$, whence it can be written as
$$2\langle\boldsymbol x,\boldsymbol u_i\rangle\langle\boldsymbol x,\boldsymbol v_i\rangle$$
for some $\boldsymbol v_i\in\mathbb R^{n-1}$. It follows that for every $i=1,2,\ldots,L$, there exists $\boldsymbol v_i\in\mathbb R^{n-1}$ such that
\begin{equation}\label{eq:16}D^3f[\boldsymbol u_i,\boldsymbol a,\boldsymbol b]=\lambda_i\langle\boldsymbol a,\boldsymbol b\rangle+\langle\boldsymbol a,\boldsymbol u_i\rangle\langle\boldsymbol b,\boldsymbol v_i\rangle+\langle\boldsymbol a,\boldsymbol v_i\rangle\langle\boldsymbol b,\boldsymbol u_i\rangle \end{equation}
for every $\boldsymbol a,\boldsymbol b\in\mathbb R^{n-1}$. Passing from covectors to vectors, we may equivalently write
$$D^3f[\boldsymbol u_i,\boldsymbol a,{}\cdot{}]^*=\lambda_i\boldsymbol a+\langle\boldsymbol a,\boldsymbol u_i\rangle\boldsymbol v_i+\langle\boldsymbol a,\boldsymbol v_i\rangle\boldsymbol u_i.$$
Applying the above equality to $\boldsymbol u_j$ yields
$$D^3f[\boldsymbol u_i,\boldsymbol u_j,{}\cdot{}]^*=\lambda_i\boldsymbol u_j+\langle\boldsymbol u_j,\boldsymbol u_i\rangle\boldsymbol v_i+\langle\boldsymbol u_j,\boldsymbol v_i\rangle\boldsymbol u_i.$$
Since the left-hand side is symmetric, so must also be the right-hand side. This gives us
$$\lambda_i\boldsymbol u_j+\langle\boldsymbol u_j,\boldsymbol u_i\rangle\boldsymbol v_i+\langle\boldsymbol u_j,\boldsymbol v_i\rangle\boldsymbol u_i=\lambda_j\boldsymbol u_i+\langle\boldsymbol u_i,\boldsymbol u_j\rangle\boldsymbol v_j+\langle\boldsymbol u_i,\boldsymbol v_j\rangle\boldsymbol u_j,$$
which can be rewritten as
$$\langle\boldsymbol u_i,\boldsymbol u_j\rangle(\boldsymbol v_i-\boldsymbol v_j)=(\lambda_j-\langle\boldsymbol u_j,\boldsymbol v_i\rangle)\boldsymbol u_i-(\lambda_i-\langle\boldsymbol u_i,\boldsymbol v_j\rangle)\boldsymbol u_j.$$
Now, if $\langle\boldsymbol u_i,\boldsymbol u_j\rangle\neq 0$, it follows that
\begin{equation}\label{eq:09}\boldsymbol v_i-\boldsymbol v_j=\frac{\lambda_j-\langle\boldsymbol u_j,\boldsymbol v_i\rangle}{\langle\boldsymbol u_j,\boldsymbol u_i\rangle}\boldsymbol u_i-\frac{\lambda_i-\langle\boldsymbol u_i,\boldsymbol v_j\rangle}{\langle\boldsymbol u_i,\boldsymbol u_j\rangle}\boldsymbol u_j.\end{equation}
Otherwise, the left-hand side vanishes, whence
\begin{equation}\label{eq:10}\langle\boldsymbol u_j,\boldsymbol v_i\rangle=\lambda_j\quad\text{and}\quad\langle\boldsymbol u_i,\boldsymbol v_j\rangle=\lambda_i.\end{equation}

Since the conditions \eqref{eq:09} and \eqref{eq:10} are of completely different nature, the rest of the argument will depend heavily on which of them holds for which pair of vectors. It will therefore be helpful to introduce the concept of an \emph{orthogonality graph}. Let $G$ be a graph on $n$ vertices, the point light sources on $T_pM^{n-1}$, connected by an edge if and only if $\boldsymbol u_i\not\perp\boldsymbol u_j$. In the following lemma, we consider the generic case, when the vectors are pairwise non-orthogonal:

\begin{lemma}\thlabel{lem:02}
Let $n\geq 4$ and suppose that $G$ contains a complete graph $K_{n+1}$. Then, in a suitable coordinate system, $D^3f=\mathbf 0$.
\end{lemma}

\begin{proof}
Without loss of generality we may assume that $u_1,u_2,\ldots,u_{n+1}$ are the vertices of $K_{n+1}$. By \eqref{eq:09}, we have
$$\boldsymbol v_i-\boldsymbol v_1=\mu_i\boldsymbol u_i-\nu_i\boldsymbol u_1$$
for every $i\neq 1$, whence
$$\boldsymbol v_i-\boldsymbol v_j=\mu_i\boldsymbol u_i-\mu_j\boldsymbol u_j-(\nu_i-\nu_j)\boldsymbol u_1$$
for every $i,j\neq 1$. Because $\dim T_pM^{n-1}\geq 3$ and the light sources are in a general position, $\boldsymbol u_1$ does not lie in the subspace $\Span\{\boldsymbol u_i,\boldsymbol u_j\}$ spanned by $\boldsymbol u_i$ and $\boldsymbol u_j$. Again by \eqref{eq:09}, it follows that $\nu_i-\nu_j=0$ for every $i,j\neq 1$, and thus $\nu_i\equalscolon\nu$ for every $i\neq 1$. Hence
\begin{equation}\label{eq:17}\boldsymbol v_i=\mu_i\boldsymbol u_i+\boldsymbol v_1-\nu\boldsymbol u_1\equalscolon\mu_i\boldsymbol u_i+\boldsymbol w.\end{equation}
Furthermore, by defining $\mu_1\colonequals\nu$, we can make \eqref{eq:17} hold for every $i=1,2,\ldots,n+1$. Substituting \eqref{eq:17} back to \eqref{eq:09} yields
$$\mu_i=\frac{\lambda_j-\langle\boldsymbol u_j,\mu_i\boldsymbol u_i+\boldsymbol w\rangle}{\langle\boldsymbol u_j,\boldsymbol u_i\rangle}=\frac{\lambda_j-\langle\boldsymbol u_j,\boldsymbol w\rangle}{\langle\boldsymbol u_j,\boldsymbol u_i\rangle}-\mu_i$$
for every $i\neq j$, which is equivalent to
\begin{equation}\label{eq:18}\langle\boldsymbol u_j,2\mu_i\boldsymbol u_i+\boldsymbol w\rangle=\lambda_j.\end{equation}

\begin{claim}
Without loss of generality, we may now assume that $\lambda_i=0$ for every $i=1,2,\ldots,n-1$. Indeed, we may still add to the right-hand side of \eqref{eq:16} a trilinear form
$$-\langle\boldsymbol a,\boldsymbol b\rangle\langle\boldsymbol u_i,\boldsymbol\eta\rangle-\langle\boldsymbol a,\boldsymbol u_i\rangle\langle\boldsymbol b,\boldsymbol\eta\rangle-\langle\boldsymbol a,\boldsymbol\eta\rangle\langle\boldsymbol b,\boldsymbol u_i\rangle,$$
which corresponds to \eqref{eq:11}. If we choose $\boldsymbol\eta$ satisfying $\langle\boldsymbol u_i,\boldsymbol\eta\rangle=\lambda_i$ for every $i=1,2,\ldots,n-1$, i.e., $\boldsymbol\eta=\boldsymbol U^{-\top}\boldsymbol\lambda$, we obtain
$$D^3f[\boldsymbol u_i,\boldsymbol a,\boldsymbol b]=\langle\boldsymbol a,\boldsymbol u_i\rangle\langle\boldsymbol b,\boldsymbol v_i-\boldsymbol\eta\rangle+\langle\boldsymbol a,\boldsymbol v_i-\boldsymbol\eta\rangle\langle\boldsymbol b,\boldsymbol u_i\rangle$$
for every $i=1,2,\ldots,n-1$. Replacing $\lambda_i$ by $\lambda_i-\langle\boldsymbol u_i,\boldsymbol\eta\rangle$ and $\boldsymbol v_i$ by $\boldsymbol v_i-\boldsymbol\eta$ for every $i=1,2,\ldots,L$ concludes the argument.
\end{claim}

Now, since $\lambda_j=0$ for every $j=1,2,\ldots,n-1$, we have $2\mu_n\boldsymbol u_n+\boldsymbol w=\mathbf 0=2\mu_{n+1}\boldsymbol u_{n+1}+\boldsymbol w$. Indeed, both vectors are orthogonal to $\boldsymbol u_1,\boldsymbol u_2,\ldots,\boldsymbol u_{n-1}$, which span $T_pM^{n-1}$. It follows that $2\mu_n\boldsymbol u_n=-\boldsymbol w=2\mu_{n+1}\boldsymbol u_{n+1}$, whence $\mu_n,\mu_{n+1}=0$ and consequently $\boldsymbol w=\mathbf 0$. Substituting back to \eqref{eq:18} yields $2\mu_i\langle\boldsymbol u_j,\boldsymbol u_i\rangle=0$, whence $\mu_i=0$ for every $i=1,2,\ldots,n-1$. By \eqref{eq:17}, we have $\boldsymbol v_i=\mathbf 0$ for every $i=1,2,\ldots,n-1$, and further, by \eqref{eq:16}, we have $D^3f[\boldsymbol u_i,{}\cdot{},{}\cdot{}]=\mathbf 0$ for every $i=1,2,\ldots,n-1$. Because $\boldsymbol u_1,\boldsymbol u_2,\ldots,\boldsymbol u_{n-1}$ span $T_pM^{n-1}$, this concludes the proof.
\end{proof}

\begin{proof}[Proof of \thref{thm:01}]
Let $p\in M^{n-1}$ be the point attaining the maximal Euclidean distance from the origin. It means that $M^{n-1}$ is contained in some sphere tangent to $M^{n-1}$ at $p$. In particular, the second fundamental form of $M^{n-1}$ at $p$ majorizes the second fundamental form of the sphere, and thus $M^{n-1}$ is strongly convex on some open neighborhood of $p$. Let $U\subseteq M^{n-1}$ be a maximal open neighborhood of $p$ where the second fundamental form of $M^{n-1}$ is positive definite.\\

Now, let $L\colonequals R(n+1,n)$ be the Ramsey number and let $G$ be the orthogonality graph on vertices $(u_i)_{i=1}^L$. By Ramsey's Theorem, we will find either a subset of $n+1$ vectors in general position that are pairwise non-orthogonal or a subset of $n$ vectors that are pairwise orthogonal. However, since the dimension of $T_pM^{n-1}$ is less than $n$, the latter is not possible. Hence, we can apply \thref{lem:02} to conclude that the cubic form \eqref{eq:15} vanishes at $p$. It follows from \thref{thm:08} that $U$ is contained in some hyperquadric $Q^{n-1}$.\\

Finally, suppose that the boundary $\partial U$ is non-empty and let $p\in\partial U$. Since $Q^{n-1}$ is locally strongly convex, the second fundamental form of $Q^{n-1}$ at $p$ is positive definite. However, the second fundamental form of $M^{n-1}$ at $p$ must be equal to the latter, and thus it is also positive definite on some open neighborhood of $p$, which contradicts the definition of $p$. It follows that $U=M^{n-1}$, which concludes the proof.
\end{proof}

In dimension $n=3$, the above argument breaks down. This is because the vectors $\boldsymbol a,\boldsymbol b,\boldsymbol c,\boldsymbol d$ are $1$-dimensional and each of the equations \eqref{eq:06}--\eqref{eq:08}, when separated from the others, is vacuously satisfied for some values of the parameters $m,r$. Only three combined equations allow us to eliminate the values of the two parameters and obtain a non-trivial condition for the series expansion of $f$. However, for this, we need to assume higher smoothness of the boundary.

\begin{proof}[Proof of \thref{thm:02}]
In dimension $n=3$, equations \eqref{eq:06}--\eqref{eq:08} read
\begin{align}
\label{eq:19}f^{(1,2)}={}&(-m +r^{-1}),\\
\label{eq:20}f^{(1,3)}={}&(-m +2 r^{-1}) f^{(0,3)} -3 m f^{(2,1)},\\
\label{eq:21}f^{(1,4)}={}&(-m +3 r^{-1}) f^{(0,4)} -6 m f^{(2,2)} -4 m f^{(0,3)} f^{(2,1)}\\
\nonumber&\quad -3 ((-m +r^{-1}) (-4 m +r^{-1}) r^{-1} +f^{(3,0)} (m^2 -m r^{-1} +r^{-2})).
\end{align}
Suppose that $f^{(0,3)} -3 f^{(2,1)}\neq 0$. Solving \eqref{eq:19}, \eqref{eq:20} for $m,r^{-1}$ yields
$$m= -\frac{2 f^{(0,3)} f^{(1,2)}-f^{(1,3)}}{f^{(0,3)}-3 f^{(2,1)}},\quad r^{-1}= -\frac{f^{(0,3)} f^{(1,2)}+3 f^{(2,1)} f^{(1,2)}-f^{(1,3)}}{f^{(0,3)}-3 f^{(2,1)}}.$$
After substituting these values to \eqref{eq:21}, we obtain that the expression
\begin{align*}
&-8 f^{(1,2)} f^{(2,1)} f^{(0,3)3}-21 f^{(1,2)3} f^{(0,3)2}+24 f^{(1,2)} f^{(2,1)2} f^{(0,3)2}+f^{(0,4)} f^{(1,2)} f^{(0,3)2}+f^{(1,4)} f^{(0,3)2}\\
&\quad +4 f^{(1,3)} f^{(2,1)} f^{(0,3)2}-12 f^{(1,2)} f^{(2,2)} f^{(0,3)2}+9 f^{(1,2)2} f^{(3,0)} f^{(0,3)2}-12 f^{(1,3)} f^{(2,1)2} f^{(0,3)}\\
&\quad +30 f^{(1,2)2} f^{(1,3)} f^{(0,3)}-2 f^{(0,4)} f^{(1,3)} f^{(0,3)}-54 f^{(1,2)3} f^{(2,1)} f^{(0,3)}+6 f^{(0,4)} f^{(1,2)} f^{(2,1)} f^{(0,3)}\\
&\quad -6 f^{(1,4)} f^{(2,1)} f^{(0,3)}+6 f^{(1,3)} f^{(2,2)} f^{(0,3)}+36 f^{(1,2)} f^{(2,1)} f^{(2,2)} f^{(0,3)}-9 f^{(1,2)} f^{(1,3)} f^{(3,0)} f^{(0,3)}\\
&\quad -9 f^{(1,2)} f^{(1,3)2}+27 f^{(1,2)3} f^{(2,1)2}-27 f^{(0,4)} f^{(1,2)} f^{(2,1)2}+9 f^{(1,4)} f^{(2,1)2}+18 f^{(1,2)2} f^{(1,3)} f^{(2,1)}\\
&\quad +6 f^{(0,4)} f^{(1,3)} f^{(2,1)}-18 f^{(1,3)} f^{(2,1)} f^{(2,2)}+3 f^{(1,3)2} f^{(3,0)}+27 f^{(1,2)2} f^{(2,1)2} f^{(3,0)}-9 f^{(1,2)} f^{(1,3)} f^{(2,1)} f^{(3,0)}
\end{align*}
must vanish in the direction of each vector $\boldsymbol u_i$. In other words, it must vanish after composing the function $f$ with at least $L(3)$ different rotations. Such a composition with an abstract rotation by angle $t$ turns out to be a trigonometric polynomial in $t$ of degree $15$. As such, it can have at most $30$ roots, unless all its Fourier coefficients vanish. However, the leading coefficient by $e^{15it}$ is equal to
$$\frac{(i f^{(0,3)}+3 f^{(1,2)}-3 i f^{(2,1)}-f^{(3,0)})^5}{4096},$$
which by assumption is non-zero. Now, if we choose $L(3)$ to be at least $31$, it leads to a contradiction.\\

It follows that the expression $f^{(0,3)} -3 f^{(2,1)}$ must vanish in the direction of each vector $\boldsymbol u_i$. Its composition with an abstract rotation by angle $t$ is equal to
$$\frac{1}{2} (f^{(0,3)}+3 i f^{(1,2)}-3 f^{(2,1)}-i f^{(3,0)}) e^{-3it}+\frac{1}{2} (f^{(0,3)}-3 i f^{(1,2)}-3 f^{(2,1)}+i f^{(3,0)}) e^{3it},$$
which vanishes identically iff $f^{(0,3)} -3 f^{(2,1)}=0$ and $f^{(3,0)} -3 f^{(1,2)}=0$. In this case, we have
$$f(x,y)=\frac{1}{2}(x^2+y^2)+\frac{1}{6}(x^2+y^2)(x f^{(3,0)}+y f^{(0,3)})+O(\sqrt{x^2+y^2})^4.$$
Recall that we may still add to the right-hand side a cubic form \eqref{eq:11}, so that
$$f(x,y)=\frac{1}{2}(x^2+y^2)+O(\sqrt{x^2+y^2})^4.$$
Hence, the cubic form \eqref{eq:15} vanishes, and we conclude the proof exactly as in the case of \thref{thm:01}.
\end{proof}

%%%%%%%%%%%%%%%%%%%%%%%%%%%%%%%%%%%%%%%%%%%%%%%%%%%%%%%%%%%%%%%%%%%%%%
\section{Bounds on the number of point light sources}\label{sec:03}
%%%%%%%%%%%%%%%%%%%%%%%%%%%%%%%%%%%%%%%%%%%%%%%%%%%%%%%%%%%%%%%%%%%%%%

Throughout this section, we will assume that $n\geq 4$.

\subsection{Upper bound}

In the proof of \thref{thm:01} we have been quite profligate in requiring that $G$ contains a complete graph on $n+1$ vertices, but since we were not concerned with the optimal value of the constant $L(n)$, this did not matter much to us, and it significantly simplified the proof. A classical lower bound on the diagonal Ramsey numbers due to P.~Erd\H{o}s \cite[Theorem~1]{Erds1947SomeRO} yields
$$R(n+1,n)\geq 2^{n/2},$$
which is exponential in $n$. However, we do not need all of the given information. The argument naturally splits into two independent parts. In the first part, we show that if $G$ contains a subgraph $K_{n+1}$, then $\boldsymbol v_i=\mu_i\boldsymbol u_i+\boldsymbol w$ for some fixed $\boldsymbol w\in\mathbb R^{n-1}$ and every $i=1,2,\ldots,n+1$. In the second part, we conclude that, in a suitable coordinate system, $D^3f=\mathbf 0$. This last observation is completely general:

\begin{lemma}\thlabel{lem:06}
Suppose that
\begin{equation}\label{eq:24}\boldsymbol v_i=\mu_i\boldsymbol u_i+\boldsymbol w\end{equation}
holds for some fixed $\boldsymbol w\in\mathbb R^{n-1}$ and some light sources $u_1,u_2,\ldots,u_{n+1}$. Then, in a suitable coordinate system, $D^3f=\mathbf 0$.
\end{lemma}

\begin{proof}
Let $\boldsymbol u_i,\boldsymbol u_j$, $i\neq j$, be any pair of vectors. They may or may not be orthogonal. In the latter case, substituting \eqref{eq:24} back to \eqref{eq:09} yields
\begin{gather*}
\mu_i=\frac{\lambda_j-\langle\boldsymbol u_j,\mu_i\boldsymbol u_i+\boldsymbol w\rangle}{\langle\boldsymbol u_j,\boldsymbol u_i\rangle}=\frac{\lambda_j-\langle\boldsymbol u_j,\boldsymbol w\rangle}{\langle\boldsymbol u_j,\boldsymbol u_i\rangle}-\mu_i,\\
\mu_j=\frac{\lambda_i-\langle\boldsymbol u_i,\mu_j\boldsymbol u_j+\boldsymbol w\rangle}{\langle\boldsymbol u_i,\boldsymbol u_j\rangle}=\frac{\lambda_i-\langle\boldsymbol u_i,\boldsymbol w\rangle}{\langle\boldsymbol u_i,\boldsymbol u_j\rangle}-\mu_j,
\end{gather*}
which is equivalent to
\begin{equation}\label{eq:12}2\langle\boldsymbol u_j,\mu_i\boldsymbol u_i\rangle=\lambda_j-\langle\boldsymbol u_j,\boldsymbol w\rangle,\quad 2\langle\boldsymbol u_i,\mu_j\boldsymbol u_j\rangle=\lambda_i-\langle\boldsymbol u_i,\boldsymbol w\rangle.\end{equation}
In the former case, substituting \eqref{eq:24} back to \eqref{eq:10} immediately yields
\begin{equation}\label{eq:23}2\langle\boldsymbol u_j,\mu_i\boldsymbol u_i\rangle=0=\lambda_j-\langle\boldsymbol u_j,\boldsymbol w\rangle,\quad 2\langle\boldsymbol u_i,\mu_j\boldsymbol u_j\rangle=0=\lambda_i-\langle\boldsymbol u_i,\boldsymbol w\rangle.\end{equation}
In either case, it follows that
$$\mu_i\lambda_i-\langle\mu_i\boldsymbol u_i,\boldsymbol w\rangle=2\langle\mu_i\boldsymbol u_i,\mu_j\boldsymbol u_j\rangle=\mu_j\lambda_j-\langle\mu_j\boldsymbol u_j,\boldsymbol w\rangle,$$
whence
\begin{equation}\label{eq:13}\lambda_i\mu_i-\langle\mu_i\boldsymbol u_i,\boldsymbol w\rangle\equalscolon 2\xi,\quad\langle\mu_i\boldsymbol u_i,\mu_j\boldsymbol u_j\rangle=\xi \end{equation}
for every $i\neq j$.\\

In the rest of the proof, we will look separately at vectors with non-zero and zero coefficients. Define $X\colonequals\{\boldsymbol u_i:\mu_i\neq 0\}$ and $Y\colonequals\{\boldsymbol u_i:\mu_i=0\}$.

\begin{claim}
Without loss of generality, we may assume that $|X|\leq n$. Indeed, suppose the contrary and define the $(n-1)\times(n+1)$ matrix
$$\boldsymbol U\colonequals\begin{pmatrix}\mu_1\boldsymbol u_1,&\mu_2\boldsymbol u_2,&\ldots,&\mu_{n+1}\boldsymbol u_{n+1}\end{pmatrix},$$
where $\boldsymbol u_1,\boldsymbol u_2,\ldots,\boldsymbol u_{n+1}\in X$ are some arbitrary vectors. By \eqref{eq:13} we have
$$\boldsymbol U^\top\boldsymbol U=\begin{pmatrix}\|\mu_1\boldsymbol u_1\|^2&\xi&\cdots&\xi&\xi\\\xi&\|\mu_2\boldsymbol u_2\|^2&\cdots&\xi&\xi\\\vdots&&\ddots&&\vdots\\\xi&\xi&\cdots&\|\mu_n\boldsymbol u_n\|^2&\xi\\\xi&\xi&\cdots&\xi&\|\mu_{n+1}\boldsymbol u_{n+1}\|^2\end{pmatrix},$$
which can be rewritten as
$$\diag\begin{pmatrix}\|\mu_1\boldsymbol u_1\|^2-\xi,&\|\mu_2\boldsymbol u_2\|^2-\xi,&\ldots,&\|\mu_n\boldsymbol u_n\|^2-\xi,&\|\mu_{n+1}\boldsymbol u_{n+1}\|^2-\xi\end{pmatrix}+\xi\boldsymbol J.$$

On the one hand, $\boldsymbol U^\top\boldsymbol U$ is a Gram matrix, whence its rank equals the dimension of the column space of $\boldsymbol U$, i.e., $n-1$. On the other hand, it is a rank-one perturbation of some diagonal matrix $\boldsymbol D$. Now, if $\boldsymbol D$ has non-zero entries, then its rank equals $n+1$ and hence the rank of $\boldsymbol U^\top\boldsymbol U$ is at least $n$, a contradiction. Secondly, if $\boldsymbol D$ has exactly one zero entry, then the rank of $\boldsymbol U^\top\boldsymbol U$ equals $n+1$. Indeed, without loss of generality, we may assume that $\|\mu_{n+1}\boldsymbol u_{n+1}\|^2-\xi=0$. Subtracting the last row from every other row, we obtain a lower-triangular matrix
$$\begin{pmatrix}\|\mu_1\boldsymbol u_1\|^2-\xi&0&\cdots&0&0\\0&\|\mu_2\boldsymbol u_2\|^2-\xi&\cdots&0&0\\\vdots&&\ddots&&\vdots\\0&0&\cdots&\|\mu_n\boldsymbol u_n\|^2-\xi&0\\\xi&\xi&\cdots&\xi&\|\mu_{n+1}\boldsymbol u_{n+1}\|^2\end{pmatrix},$$
with non-zero entries on the diagonal, whence its rank equals $n+1$, a contradiction. Finally, if $\boldsymbol D$ has at least two zero entries, without loss of generality we may assume that $\|\mu_1\boldsymbol u_1\|^2-\xi=0=\|\mu_2\boldsymbol u_2\|^2-\xi$, whence $\|\mu_1\boldsymbol u_1-\mu_2\boldsymbol u_2\|^2=\xi-2\xi+\xi=0$, a contradiction.
\end{claim}

\begin{claim}
Without loss of generality, we may assume that $|Y|=n+1$. For suppose to the contrary that $|Y|\leq n$. Since $|X|+|Y|=n+1$ and $|X|\leq n$, it follows that $Y$ is non-empty. Let $\boldsymbol u_i\in Y$ be some arbitrary vector, and define $\tilde Y\colonequals Y\setminus\{\boldsymbol u_i\}$. Then \eqref{eq:12} and \eqref{eq:23} imply that
$$0=2\langle\boldsymbol u_j,\mu_i\boldsymbol u_i\rangle=\lambda_j-\langle\boldsymbol u_j,\boldsymbol w\rangle=2\langle\boldsymbol u_j,\mu_k\boldsymbol u_k\rangle$$
for every $j\neq i$ and $k\neq j$, whence $\boldsymbol u_k\perp\boldsymbol u_j$ for every $\boldsymbol u_k\in X$ and every other $\boldsymbol u_j\in X\cup\tilde Y$. In particular, we have $\Span X\perp\Span\tilde Y$. Further, the elements of $X$ are pairwise orthogonal and thus linearly independent. Because the light sources are in a general position and, by assumption, $|\tilde Y|\leq n-1$, the elements of $\tilde Y$ are likewise linearly independent. It follows that $n=|X|+|\tilde Y|=\dim\Span X+\dim\Span\tilde Y=\dim(\Span X\oplus\Span\tilde Y)\leq n-1$, a contradiction.
\end{claim}

Finally, let $\boldsymbol u_1,\boldsymbol u_2,\ldots,\boldsymbol u_{n-1}\in Y$ be some arbitrary vectors. By \eqref{eq:12} and \eqref{eq:23} we have $\lambda_i=\langle\boldsymbol u_i,\boldsymbol w\rangle$ for every $i=1,2,\ldots,n-1$. Substituting this back to \eqref{eq:16} yields
$$D^3f[\boldsymbol u_i,\boldsymbol a,\boldsymbol b]=\langle\boldsymbol u_i,\boldsymbol w\rangle\langle\boldsymbol a,\boldsymbol b\rangle+\langle\boldsymbol a,\boldsymbol u_i\rangle\langle\boldsymbol b,\boldsymbol w\rangle+\langle\boldsymbol a,\boldsymbol w\rangle\langle\boldsymbol b,\boldsymbol u_i\rangle.$$
Since the vectors $\boldsymbol u_1,\boldsymbol u_2,\ldots,\boldsymbol u_{n-1}$ span $T_pM^{n-1}$, the above equality holds for every $\boldsymbol u,\boldsymbol a,\boldsymbol b\in T_pM^{n-1}$. In particular, it follows that $D^3f[\boldsymbol x,\boldsymbol x,\boldsymbol x]=3\langle\boldsymbol x,\boldsymbol x\rangle\langle\boldsymbol x,\boldsymbol w\rangle$. By taking $\boldsymbol\eta\colonequals 6\boldsymbol w$ in \eqref{eq:11}, we conclude the proof.
\end{proof}

If we are sufficiently careful, we can significantly weaken the assumption in the first part of our argument:

\begin{lemma}
Suppose that $G$ admits a proper ear decomposition with ears of vertex-length at most $n-1$. Then \eqref{eq:24} holds for some fixed $\boldsymbol w\in\mathbb R^{n-1}$ and every $u_i\in V(G)$.
\end{lemma}

\begin{proof}
Let $u_1,u_2,\ldots,u_m$ be the first ear in the sequence, which is a cycle in $G$. By \eqref{eq:09}, we have
$$\boldsymbol v_{i+1}-\boldsymbol v_i=\mu_{i+1}\boldsymbol u_{i+1}-\nu_i\boldsymbol u_i,\quad i=1,2,\ldots,m.$$
Cyclic summation of both sides yields
$$\mathbf 0=\sum_{i=1}^m(\mu_i-\nu_i)\boldsymbol u_i,$$
whence $\mu_i=\nu_i$ for every $i=1,2,\ldots,m$. Indeed, because the light sources are in a general position, any $m\leq n-1$ vectors $\boldsymbol u_i$ are linearly independent. Thus, we can write
$$\boldsymbol v_i-\boldsymbol v_1=\mu_i\boldsymbol u_i-\mu_1\boldsymbol u_1,$$
whence
$$\boldsymbol v_i=\mu_i\boldsymbol u_i-\mu_1\boldsymbol u_1+\boldsymbol v_1\equalscolon\mu_i\boldsymbol u_i+\boldsymbol w$$
for every $i=1,2,\ldots,m$.\\

We will prove by induction on the number of ears that \eqref{eq:24} holds for every vertex $v\in V(G)$, with the same translation $\boldsymbol w$. Denote by $H\subset G$ the graph constructed so far, and let $u_{m+1},u_{m+2},\ldots,u_{m+k}$ be the next ear in the sequence, with $u_{m+1},u_{m+k}\in V(H)$. From the inductive hypothesis, we immediately get
\begin{equation}\label{eq:22}\boldsymbol v_{m+1}-\boldsymbol v_{m+k}=\mu_{m+1}\boldsymbol u_{m+1}-\mu_{m+k}\boldsymbol u_{m+k}\equalscolon\tilde\mu_{m+1}\boldsymbol u_{m+1}-\tilde\nu_{m+k}\boldsymbol u_{m+k},\end{equation}
and the same argument as for the base case yields
$$\boldsymbol v_i=\tilde\mu_i\boldsymbol u_i+\tilde{\boldsymbol w}$$
for every $i=m+1,m+2,\ldots,m+k$. Further, \eqref{eq:22} implies that $\tilde\mu_{m+1}=\mu_{m+1}$, whence also $\tilde{\boldsymbol w}=\boldsymbol w$. By the principle of mathematical induction, this concludes the proof.
\end{proof}

The above result seems difficult to improve upon. The problem arises when the graph $G$ is too sparse, i.e., too many pairs of vectors $\boldsymbol u_i$ are orthogonal. Already in the case where $G$ is a cycle on $n+1$ vertices, the conclusion ceases to be true. Moreover, the corresponding cubic form \eqref{eq:15} does not need to possess any symmetry. However, we know that the graph $G$ must actually be quite dense:

\begin{lemma}\thlabel{lem:03}
Every induced subgraph of $G$ on $n$ vertices is connected.
\end{lemma}

\begin{proof}
Suppose, to the contrary, that there exists an induced subgraph $H\subseteq G$ on $n$ vertices that is disconnected, and let $C\subsetneq H$ be any connected component. Since $|V(C)|\leq n-1$, the corresponding vectors $\boldsymbol u_i$ are linearly independent, and since $|V(H\setminus C)|=n-|V(C)|\leq n-1$, the corresponding vectors $\boldsymbol u_i$ are likewise linearly independent. Moreover, since there is no edge in $G$ between $C$ and $H\setminus C$, $\Span\{\boldsymbol u_i:u_i\in V(C)\}$ is orthogonal to $\Span\{\boldsymbol u_i:u_i\in V(H\setminus C)\}$, whence
$n=|V(H)|=|V(C)|+|V(H\setminus C)|=\dim\Span\{\boldsymbol u_i:u_i\in V(C)\}+\dim\Span\{\boldsymbol u_i:u_i\in V(H\setminus C)\}=\dim(\Span\{\boldsymbol u_i:u_i\in V(C)\}\oplus\Span\{\boldsymbol u_i:u_i\in V(H\setminus C)\})\leq n-1$, a contradiction. This concludes the proof.
\end{proof}

The above result is closely related to the notion of vertex-connectivity. If $G$ is a graph on $n+1$ vertices, then \thref{lem:03} means precisely that it is $2$-connected. Now, a classical theorem of H.~Whitney \cite[Theorem~19]{40eb8360-6875-366c-a9fc-9bc8aa35535c} asserts that a graph is $2$-connected if and only if it admits a proper ear decomposition. Unfortunately, it may still admit no proper ear decomposition with ears of vertex-length at most $n-1$, as illustrated by the counterexample of the cycle graph $C_{n+1}$.\\

\begin{figure}
\begin{subfigure}{\linewidth}
\begin{tikzpicture}
\begin{scope}[xshift=-5cm]
\useasboundingbox (-3,-2.75) rectangle (3,2.5);
\node (A) at (-0.875,-1.75) {};
\node (B) at (-0.875,0) {};
\node (C) at (-0.875,1.75) {};
\node (D) at (0.875,-1.75) {};
\node (E) at (0.875,0) {};
\node (F) at (0.875,1.75) {};
\draw (A) -- (D) -- (B) -- (E) -- (C) -- (F) -- (A) -- (E) (B) -- (F) (C) -- (D);
\node[tag] at (0,-2.5) {$K_{3,3}\cong\overline{C_3\cup C_3}$};
\end{scope}
\begin{scope}
\useasboundingbox (-3,-2.75) rectangle (3,2.5);
\begin{scope}[rotate=90]
\node (A) at (0,0) {};
\node (B) at (0:1.75) {};
\node (C) at (72:1.75) {};
\node (D) at (144:1.75) {};
\node (E) at (216:1.75) {};
\node (F) at (288:1.75) {};
\draw (A) -- (B) -- (C) -- (A) -- (D) -- (E) -- (A) -- (F) -- (B) (C) -- (D) (E) -- (F);
\end{scope}
\node[tag] at (0,-2.5) {$W_6\cong\overline{P_1\cup C_5}$};
\end{scope}
\begin{scope}[xshift=5cm]
\useasboundingbox (-3,-2.75) rectangle (3,2.5);
\begin{scope}[rotate=90]
\node (A) at (0:0.5) {};
\node (B) at (120:0.5) {};
\node (C) at (240:0.5) {};
\node (D) at (0:2) {};
\node (E) at (120:2) {};
\node (F) at (240:2) {};
\draw (A) -- (D) -- (E) -- (B) -- (C) -- (F) -- (D) (C) -- (A) -- (B) (E) -- (F);
\end{scope}
\node[tag] at (0,-2.5) {$Y_3\cong\overline{C_6}$};
\end{scope}
\end{tikzpicture}
\caption{Minimal $3$-connected graphs on $6$ vertices}
\label{fig:03}
\end{subfigure}
\begin{subfigure}{\linewidth}
\begin{tikzpicture}
\begin{scope}[xshift=-5cm]
\useasboundingbox (-3,-2.75) rectangle (3,4.25);
\node (A) at (-0.875,-1.75) {};
\node (B) at (-0.875,0) {};
\node (C) at (-0.875,1.75) {};
\node (D) at (0.875,-1.75) {};
\node (E) at (0.875,-0) {};
\node (F) at (0.875,1.75) {};
\node (G) at (0.875,3.5) {};
\draw (A) -- (D) -- (B) -- (E) -- (A) -- (F) -- (B) -- (G) -- (C) -- (F) (D) -- (C) -- (E) (A) -- (G);
\node[tag] at (0,-2.5) {$K_{3,4}$};
\end{scope}
\begin{scope}
\useasboundingbox (-3,-2.75) rectangle (3,2.5);
\begin{scope}[rotate=90]
\node (A) at (0,0) {};
\node (B) at (0:1.75) {};
\node (C) at (60:1.75) {};
\node (D) at (120:1.75) {};
\node (E) at (180:1.75) {};
\node (F) at (240:1.75) {};
\node (G) at (300:1.75) {};
\draw (A) -- (B) -- (C) -- (A) -- (D) -- (E) -- (A) -- (F) -- (G) -- (A) (C) -- (D) (E) -- (F) (G) -- (B);
\end{scope}
\node[tag] at (0,-2.5) {$W_7$};
\end{scope}
\begin{scope}[xshift=5cm]
\useasboundingbox (-3,-2.75) rectangle (3,2.5);
\begin{scope}[rotate=90]
\node (A) at (0:0.5) {};
\node (B) at (120:0.5) {};
\node (C) at (240:0.5) {};
\node (D) at (0:2) {};
\node (E) at (120:2) {};
\node (F) at (240:2) {};
\node (G) at ($(B)!.5!(C)$) {};
\draw (A) -- (D) -- (E) -- (B) -- (G) -- (C) -- (F) -- (D) (C) -- (A) -- (B) (E) -- (F);
\draw[dashed] (G) -- (A);
\draw[dashed] (G) to[bend left=20] (E);
\draw[dashed] (G) to[bend left=20] (D);
\end{scope}
\node[tag] at (0,-2.5) {$Y_3[e]+e$};
\end{scope}
\end{tikzpicture}
\caption{Minimal $3$-connected graphs on $7$ vertices}
\label{fig:04}
\end{subfigure}
\caption{List of specific graphs used in the proof of \thref{lem:04}}
\end{figure}

If $G$ is a graph on $n+2$ vertices, then \thref{lem:03} means precisely that it is $3$-connected. For $n=4$, every $3$-connected graph on $6$ vertices contains either the wheel graph $W_6$, or the complete bipartite graph $K_{3,3}$, or the triangular prism graph $Y_3$ (fig.~\ref{fig:03}). The latter two still admit no proper ear decomposition with ears of vertex-length at most $3$. However, for $n=5$, every $3$-connected graph on $7$ vertices contains either the wheel graph $W_7$, or the complete bipartite graph $K_{3,4}$, or the triangular prism graph $Y_3$ with a vertex of one $3$-cycle bridged to an edge of the other $3$-cycle (fig.~\ref{fig:04}). Since each of the five graphs clearly admits a proper ear decomposition with ears of vertex-length at most $4$, the following lemma holds for $n=5$:

\begin{lemma}\thlabel{lem:04}
Let $n\geq 5$ and suppose that $G$ is a $3$-connected graph on $n+2$ vertices. Then it admits a proper ear decomposition with ears of vertex-length at most $n-1$.
\end{lemma}

\begin{proof}
We will prove \thref{lem:04} by induction on the number of vertices of $G$. The base case $n=5$ can be manually verified by an exhaustive case analysis.\\

We say that a $3$-connected graph $G$ on $n+2$ vertices has property $\phi_{-3}$ if it admits a proper ear decomposition with ears of vertex-length at most $n-1$. W.T.~Tutte \cite[\S 5]{TUTTE1961441} proved that a graph is $3$-connected if and only if it is a wheel graph or is obtained from a wheel graph by
\begin{enumerate*}[label=(\Roman*)]
\item\label{it:01} adjoining new edges whose ends are two distinct non-adjacent vertices; and
\item\label{it:02} splitting vertices incident with $4$ or more edges, and adjoining an edge incident with the two resulting new vertices (fig.~\ref{fig:05}).
\end{enumerate*}
Hence, for the inductive step, it suffices to show that every wheel graph has property $\phi_{-3}$, and that both operations \ref{it:01} and \ref{it:02} preserve property $\phi_{-3}$.\\

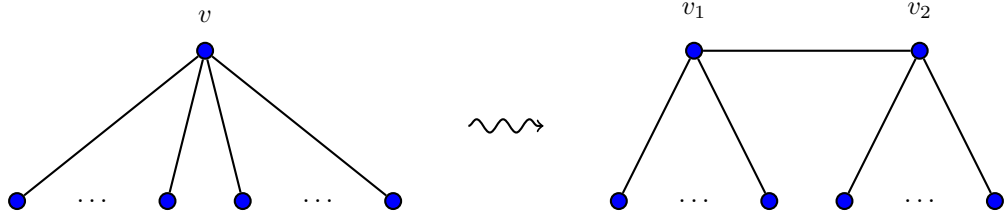
\begin{figure}
\begin{tikzpicture}
\begin{scope}[xshift=-4cm]
\node[label={[label distance=6pt]above:$v$}] (A) at (0,2) {};
\node (B) at (-2.5,0) {};
\node[tag] at (-1.5,0) {$\ldots$};
\node (C) at (-0.5,0) {};
\node (D) at (0.5,0) {};
\node[tag] at (1.5,0) {$\ldots$};
\node (E) at (2.5,0) {};
\draw (A) -- (B) (A) -- (C) (A) -- (D) (A) -- (E);
\end{scope}
\draw[->, decorate, decoration = snake] (-0.5,1) -- (0.5,1);
\begin{scope}[xshift=4cm]
\node[label={[label distance=6pt]above:$v_1$}] (A1) at (-1.5,2) {};
\node[label={[label distance=6pt]above:$v_2$}] (A2) at (1.5,2) {};
\node (B) at (-2.5,0) {};
\node[tag] at (-1.5,0) {$\ldots$};
\node (C) at (-0.5,0) {};
\node (D) at (0.5,0) {};
\node[tag] at (1.5,0) {$\ldots$};
\node (E) at (2.5,0) {};
\draw (A1) -- (A2) (A1) -- (B) (A1) -- (C) (A2) -- (D) (A2) -- (E);
\end{scope}
\end{tikzpicture}
\caption{Operation \ref{it:02} of splitting a vertex incident with $4$ or more edges}
\label{fig:05}
\end{figure}

Since every wheel graph admits a proper ear decomposition with ears of vertex-length at most $3$, the first claim follows immediately for $n\geq 4$. Further, operation \ref{it:01} involves attaching a single-edge ear of vertex-length $2$ to the graph, which also clearly preserves property $\phi_{-3}$ for $n\geq 3$.\\

Finally, suppose that $G$ has property $\phi_{-3}$ and denote by $G'$ the graph obtained from $G$ by splitting some vertex $v\in V(G)$ into $v_1,v_2\in V(G')$. Let $\mathcal P=(P_1,P_2,\ldots)$ be a proper ear decomposition of $G$ that serves as a witness to property $\phi_{-3}$. A proper ear decomposition $\mathcal P'$ of $G'$ that serves as a witness to property $\phi_{-3}$ may be constructed as follows. If an ear $P_i\in\mathcal P$ does not contain $v$, we simply append it to $\mathcal P'$. The vertex $v$ will appear for the first time as an internal vertex of some ear $P_i=(\ldots,v,\ldots)$. Now, if the two edges incident with $v$ in $P_i$ correspond to edges incident with two different vertices $v_1,v_2$ in $G'$, we replace $v$ by $v_1,v_2$ and append $P_i'=(\ldots,v_1,v_2,\ldots)$ to $\mathcal P'$. Note that although the path length increased by $1$, the number of vertices in the graph also increased by $1$, so this does not affect property $\phi_{-3}$. From now on, the vertex $v$ will appear only as an endpoint of some ears $P_j$, $j>i$, in which case we simply replace it by $v_1$ or $v_2$, depending on the incident edge, and append the modified ear to $\mathcal P'$. Finally, if the two edges incident with $v$ in $P_i$ correspond to edges incident with the same vertex $v_1$ in $G'$, we simply replace $v$ by $v_1$ and append the modified ear to $\mathcal P'$. We continue exactly as before, until some modified ear $P_j'$, $j>i$, ends with $v_2$. In this case, we extend $P_j'$ by adding $v_1$, thereby introducing $v_2$ as an internal vertex of $P_j'$, and append the modified ear to $\mathcal P'$. Again, the path length increased by $1$, but this does not affect property $\phi_{-3}$. This concludes the proof.
\end{proof}

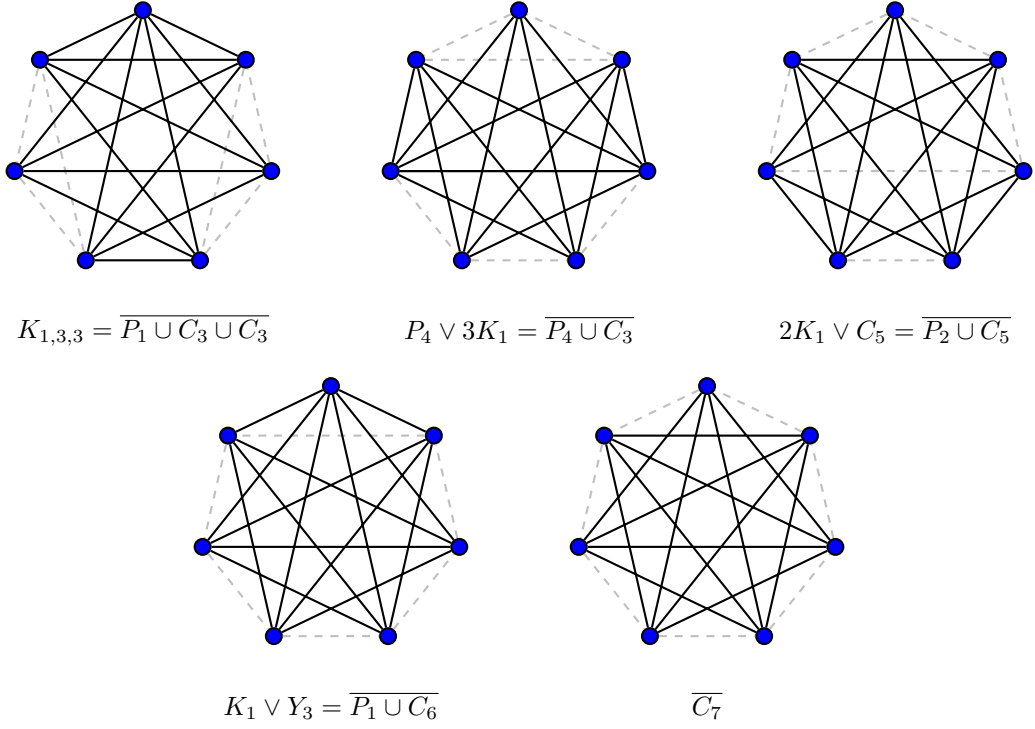
\begin{figure}
\begin{tikzpicture}
\begin{scope}[xshift=-5cm]
\useasboundingbox (-3,-2.75) rectangle (3,2.5);
\begin{scope}[rotate=90]
\node (A) at (360/7*0:1.75) {};
\node (B) at (360/7*1:1.75) {};
\node (C) at (360/7*2:1.75) {};
\node (D) at (360/7*3:1.75) {};
\node (E) at (360/7*4:1.75) {};
\node (F) at (360/7*5:1.75) {};
\node (G) at (360/7*6:1.75) {};
\draw[dashed, lightgray] (B) -- (C) -- (D) -- (B) (E) -- (F) -- (G) -- (E);
\draw (A) -- (B) -- (G) -- (C) -- (E) -- (B) -- (F) -- (D) -- (G) -- (A) -- (C) -- (F) -- (A) -- (D) -- (E) -- (A);
\end{scope}
\node[tag] at (0,-2.5) {$K_{1,3,3}=\overline{P_1\cup C_3\cup C_3}$};
\end{scope}
\begin{scope}
\useasboundingbox (-3,-2.75) rectangle (3,2.5);
\begin{scope}[rotate=90]
\node (A) at (360/7*0:1.75) {};
\node (B) at (360/7*1:1.75) {};
\node (C) at (360/7*2:1.75) {};
\node (D) at (360/7*3:1.75) {};
\node (E) at (360/7*4:1.75) {};
\node (F) at (360/7*5:1.75) {};
\node (G) at (360/7*6:1.75) {};
\draw[dashed, lightgray] (A) -- (B) -- (G) -- (A) (C) -- (D) -- (E) -- (F);
\draw (A) -- (C) -- (B) -- (E) -- (G) -- (D) -- (B) -- (F) -- (G) -- (C)-- (F) -- (A) -- (D) -- (F) -- (C) -- (E) -- (A);
\end{scope}
\node[tag] at (0,-2.5) {$P_4\vee 3K_1=\overline{P_4\cup C_3}$};
\end{scope}
\begin{scope}[xshift=5cm]
\useasboundingbox (-3,-2.75) rectangle (3,2.5);
\begin{scope}[rotate=90]
\node (A) at (360/7*0:1.75) {};
\node (B) at (360/7*1:1.75) {};
\node (C) at (360/7*2:1.75) {};
\node (D) at (360/7*3:1.75) {};
\node (E) at (360/7*4:1.75) {};
\node (F) at (360/7*5:1.75) {};
\node (G) at (360/7*6:1.75) {};
\draw[dashed, lightgray] (A) -- (B) -- (C) -- (F) -- (G) -- (A) (D) -- (E);
\draw (E) -- (A) -- (C) -- (G) -- (B) -- (F) -- (A) -- (D) -- (C) -- (E) -- (F) -- (D) -- (G) -- (E) -- (B) -- (D);
\end{scope}
\node[tag] at (0,-2.5) {$2K_1\vee C_5=\overline{P_2\cup C_5}$};
\end{scope}
\begin{scope}[shift={(-2.5cm,-5cm)}]
\useasboundingbox (-3,-2.75) rectangle (3,2.5);
\begin{scope}[rotate=90]
\node (A) at (360/7*0:1.75) {};
\node (B) at (360/7*1:1.75) {};
\node (C) at (360/7*2:1.75) {};
\node (D) at (360/7*3:1.75) {};
\node (E) at (360/7*4:1.75) {};
\node (F) at (360/7*5:1.75) {};
\node (G) at (360/7*6:1.75) {};
\draw[dashed, lightgray] (B) -- (C) -- (D) -- (E) -- (F) -- (G) -- (B);
\draw (A) -- (B) -- (D) -- (F) -- (A) -- (G) -- (E) -- (C) -- (G) -- (D) -- (A) -- (E) -- (B) -- (F) -- (C) -- (A);
\end{scope}
\node[tag] at (0,-2.5) {$K_1\vee Y_3=\overline{P_1\cup C_6}$};
\end{scope}
\begin{scope}[shift={(2.5cm,-5cm)}]
\useasboundingbox (-3,-2.75) rectangle (3,2.5);
\begin{scope}[rotate=90]
\node (A) at (360/7*0:1.75) {};
\node (B) at (360/7*1:1.75) {};
\node (C) at (360/7*2:1.75) {};
\node (D) at (360/7*3:1.75) {};
\node (E) at (360/7*4:1.75) {};
\node (F) at (360/7*5:1.75) {};
\node (G) at (360/7*6:1.75) {};
\draw[dashed, lightgray] (A) -- (B) -- (C) -- (D) -- (E) -- (F) -- (G) -- (A);
\draw (A) -- (C) -- (E) -- (G) -- (B) -- (D) -- (F) -- (A) -- (D) -- (G) -- (C) -- (F) -- (B) -- (E) -- (A);
\end{scope}
\node[tag] at (0,-2.5) {$\overline{C_7}$};
\end{scope}
\end{tikzpicture}
\caption{Minimal $4$-connected graphs on $7$ vertices}
\label{fig:06}
\end{figure}

As we have already observed, for $n=4$, it turns out that $6$ vertices are too few for \thref{lem:03} to force property $\phi_{-3}$. To obtain any bound, we must therefore consider graphs on at least $7$ vertices. In this case, \thref{lem:03} means precisely that such a graph is $4$-connected. Now, since each of the five minimal $4$-connected graphs on $7$ vertices (fig.~\ref{fig:06}) clearly admits a proper ear decomposition with ears of vertex-length at most $3$, the following lemma holds:

\begin{lemma}
Suppose that $G$ is a $4$-connected graph on $7$ vertices. Then it admits a proper ear decomposition with ears of vertex-length at most $3$.
\end{lemma}

Unfortunately, the graph method does not apply to the case $n=3$, because the key equalities \eqref{eq:09} and \eqref{eq:10} do not necessarily hold. Therefore, ultimately, for
$$L(n)=\begin{cases}31&n=3\\7&n=4\\n+2&n\geq 5\end{cases}$$
it follows from the above considerations that, in a suitable coordinate system, $D^3f=\mathbf 0$. The same argument as in the proof of \thref{thm:01} yields that $K$ is an ellipsoid.

\subsection{Lower bound}

On the other hand, if $L(n)=n$, $n\geq 3$, \thref{thm:01} does not need to hold. Interestingly, we do not have to look far for a counterexample, as the well-known unit ball of $\ell^p_n$, $p>1$, turns out to be one. Define
$$B_p^n\colonequals\{x\in\mathbb R^n:\|x\|_p\leq 1\}$$
and consider any point $\boldsymbol x\colonequals(x_1,x_2,\ldots,x_n)$ on the boundary. The normal vector at this point is given by $\boldsymbol n\colonequals(p|x_1|^{p-1}\sgn x_1,p|x_2|^{p-1}\sgn x_2,\ldots,p|x_n|^{p-1}\sgn x_n)$. Now, if the light source is placed at $\lambda\hat{\boldsymbol e}_i$, then the shadow boundary consists of points satisfying $\boldsymbol n\perp\boldsymbol x-\lambda\hat{\boldsymbol e}_i$, which is equivalent to
$$p=p\|\boldsymbol x\|_p^p=\boldsymbol n\cdot\boldsymbol x=\boldsymbol n\cdot\lambda\hat{\boldsymbol e}_i=\lambda p|x_i|^{p-1}\sgn x_i.$$
In particular, the shadow boundary is contained in the hyperplane $x_i=|\lambda|^{1/(1-p)}\sgn\lambda$. Moreover, every point $\boldsymbol x\in\partial B_p^n$ belongs to $n$ flat shadow boundaries created by light sources placed at $(|x_i|^{1-p}\sgn x_i)\hat{\boldsymbol e}_i$, $i=1,2,\ldots,n$. If $x_i=0$, then the corresponding light source is placed at infinity.\\

Observe that in the key \thref{lem:06} we need precisely $n+1$ light sources to conclude that, in a suitable coordinate system, $D^3f=\mathbf 0$. The additional light source served purely technical purposes, to ensure that we are able to find a subset that consists of $n+1$ light sources in a sufficiently generic position. Keeping all this in mind, and on the other hand, having a counterexample showing that $n$ light sources are not enough, the following conjecture seems promising to us:

\begin{conjecture}
Let $K\subset\mathbb R^n$, $n\geq 3$, be a convex body with boundary of class $C^3$. Suppose that for every point $p\in\partial K$ on the boundary, there are at least $n+1$ point light sources on the tangent hyperplane $T_p\partial K$ in a general linear position with respect to $p$, that create flat shadow boundaries on $K$. Then $K$ is an ellipsoid.
\end{conjecture}

\backmatter

%%%%%%%%%%%%%%%%%%%%%%%%%%%%%%%%%%%%%%%%%%%%%%%%%%%%%%%%%%%%%%%%%%%%%%
\section{Concluding remarks}
%%%%%%%%%%%%%%%%%%%%%%%%%%%%%%%%%%%%%%%%%%%%%%%%%%%%%%%%%%%%%%%%%%%%%%

In this section, we have collected a number of related remarks that may be of interest but would distract from the main exposition if they were placed throughout the paper.

\subsection{sec:01}

The initial smoothness assumption in both \thref{thm:01,thm:02} is crucial to our argument. But with a little effort, in \thref{thm:01} we can slightly relax it by assuming that the boundary of $K$ is of class $C^{2,1}$. Indeed, since the second fundamental form is defined everywhere, we can still find a point at which it is positive definite, and use the same argument to obtain that the cubic form \eqref{eq:15} vanishes almost everywhere on the (open) set of all the points where the second fundamental form is positive definite. By the revised \thref{thm:08}, this concludes the proof. The question remains open as to how far we can reduce the initial smoothness assumption.\\

The assumption of a general position of the light sources seems natural, and it cannot be completely dispensed with. Indeed, following the idea of A.~Marchaud \cite{Marchaud1959}, let us consider light sources lying on a hyperplane $H$ intersecting the interior of $K$. Then, for every point $p\in\partial K\setminus H$, the boundary of $K$ is locally contained in a quadric, but for $p\in\partial K\cap H$, all the light sources are contained in a lower-dimensional subspace of the tangent hyperplane, and it is these points that form the seam connecting the two quadrics.\\

An extreme case of a non-general position of light sources with respect to $p$ is when some three of them are collinear with $p$. However, at the cost of increased smoothness, this particular situation poses no difficulty, since the following lemma holds:

\begin{lemma}
Let $K\subset\mathbb R^n$, $n\geq 3$, be a convex body with boundary of class $C^5$. Suppose that the points $p,u_1,u_2,u_3$ are collinear. Then, in a suitable coordinate system, $D^3f=\mathbf 0$.
\end{lemma}

\begin{proof}
After applying a suitable orthogonal change of coordinates, we may further assume that $u_i=r_i\hat{\boldsymbol e}_1$, $i=1,2,3$. From \eqref{eq:06} we immediately obtain that $-m_i +r_i^{-1}$ does not depend on $i=1,2,3$. Denote this common value by $\lambda\in\mathbb R$, and observe that the values of $m_i$ must therefore be pairwise different. Further, \eqref{eq:06} and \eqref{eq:07} read
\begin{align}
\notag D^2f_1[\boldsymbol a,\boldsymbol b]={}&\lambda \langle\boldsymbol a,\boldsymbol b\rangle,\\
\label{eq:07'}D^3f_1[\boldsymbol a,\boldsymbol b,\boldsymbol c]={}&(m +2 \lambda) D^3f[\boldsymbol a,\boldsymbol b,\boldsymbol c] -m \sum Df_{11}[\boldsymbol a] \langle\boldsymbol b,\boldsymbol c\rangle.
\end{align}
Now, since the right-hand side of \eqref{eq:07'} is an affine function of $m$ that attains the same value for three different values of $m$, it must be constant, whence
$$D^3f[\boldsymbol a,\boldsymbol b,\boldsymbol c]=\sum Df_{11}[\boldsymbol a] \langle\boldsymbol b,\boldsymbol c\rangle.$$
Note that this does not yet complete the proof, as the above equality holds only for $\boldsymbol a,\boldsymbol b,\boldsymbol c\in\Span\{\hat{\boldsymbol e}_1\}^\perp$. Finally, \eqref{eq:08} reads
\begin{align}
\label{eq:08'}D^4f_1[\boldsymbol a,\boldsymbol b,\boldsymbol c,\boldsymbol d]={}&(2 m +3 \lambda) D^4f[\boldsymbol a,\boldsymbol b,\boldsymbol c,\boldsymbol d] -m \sum D^2f_{11}[\boldsymbol a,\boldsymbol b] \langle\boldsymbol c,\boldsymbol d\rangle -m \sum D^3f[\boldsymbol a,\boldsymbol b,\boldsymbol c] Df_{11}[\boldsymbol d]\\
\nonumber&\quad -(\lambda (-3 m +\lambda) (m +\lambda) +f_{111} (m^2 +m \lambda +\lambda^2)) \sum \langle\boldsymbol a,\boldsymbol b\rangle \langle\boldsymbol c,\boldsymbol d\rangle.
\end{align}
This time, since the right-hand side of \eqref{eq:08'} is a second-order polynomial in $m$ that attains the same value for three different values of $m$, it likewise must be constant. Equating its leading coefficient to zero yields
$$-3\lambda+f_{111}=0.$$
So, in the end, we have
\begin{align*}
D^3f[\boldsymbol x,\boldsymbol x,\boldsymbol x]&=f_{111}x_1^3+3Df_{11}[\boldsymbol x_{-1}]x_1^2+3D^2f_1[\boldsymbol x_{-1},\boldsymbol x_{-1}]x_1+D^3f[\boldsymbol x_{-1},\boldsymbol x_{-1},\boldsymbol x_{-1}]\\
&=f_{111}x_1^3+3Df_{11}[\boldsymbol x_{-1}]x_1^2+f_{111}\langle\boldsymbol x_{-1},\boldsymbol x_{-1}\rangle x_1+3Df_{11}[\boldsymbol x_{-1}]\langle\boldsymbol x_{-1},\boldsymbol x_{-1}\rangle\\
&=(x_1^2+\langle\boldsymbol x_{-1},\boldsymbol x_{-1}\rangle)(f_{111}x_1+3Df_{11}[\boldsymbol x_{-1}])=\langle\boldsymbol x,\boldsymbol x\rangle(f_{111}x_1+3Df_{11}[\boldsymbol x_{-1}]).
\end{align*}
By taking $\boldsymbol\eta\colonequals 2(f_{111}\hat{\boldsymbol e}_1+3\nabla_{-1}f_{11})$ in \eqref{eq:11}, we conclude the proof.
\end{proof}

\begin{remark}
Note that in the counterexample mentioned earlier, where $K=B_p^n$ is the unit ball of $\ell^p_n$, through every point on every coordinate hyperplane, there exists a tangent line on which two light sources lie.
\end{remark}

Let us conclude this section with a proposal for a far-reaching generalization of \thref{thm:01}. In \cite{Gonzalez_Garcia2022}, I.~Gonz\'alez-Garc\'ia, J.~Jer\'onimo-Castro, E.~Morales-Amaya, and D.J.~Verdusco-Hern\'andez proved the following result:

\begin{theorem}[{cf. \cite[Theorem~6]{Gonzalez_Garcia2022}}]\thlabel{thm:09}
Let $K,L\subset\mathbb R^n$, $n\geq 3$, be convex bodies with $L\subset\Int K$. Suppose that for every point $p\in\partial K$ on the boundary, a point light source at $p$ creates a flat shadow boundary on $L$, $L$ casts a shadow with a flat boundary on $K$, and both shadow boundaries are contained in parallel hyperplanes. Then $K,L$ are concentric, homothetic ellipsoids.
\end{theorem}

\noindent Now, observe that \thref{thm:09} is an immediate corollary from \thref{thm:01}. Indeed, the mere assumption that the shadow boundaries created on $L$ are flat is enough to conclude that $L$ is an ellipsoid. Hence, the boundaries of shadows cast by $L$ on $K$ must also be ellipsoids, and thus $K$ admits enough ellipsoidal sections that it must itself be an ellipsoid.\\

In the above argument, we used neither the assumption that the light sources are located on the boundary of $K$ nor the assumption that the shadow boundaries are contained in parallel hyperplanes. This makes us wonder if the following conjecture holds:

\begin{conjecture}\thlabel{con:06}
Let $K,L\subset\mathbb R^n$, $n\geq 3$, be convex bodies with $L\subseteq K$. Suppose that for every point $p\in\partial K$ on the boundary, there are at least $L(n)$ point light sources on the supporting cone of $L$ with apex at $p$, such that $L$ casts a shadow with flat boundary on $K$. Then $K,L$ are concentric, homothetic ellipsoids.
\end{conjecture}

\noindent Note that if $K=L$, then \thref{con:06} is equivalent to \thref{thm:01}. Indeed, in this case, supporting cones degenerate to supporting hyperplanes. \thref{con:06} is also strictly stronger than \thref{thm:09}, because we assume nothing about the location of the light sources nor about shadow boundaries created on $L$. Therefore, if true, \thref{con:06} would contain both of these results. The question of the minimum number of light sources $L(n)$ is no less interesting than before. A rather special case of \thref{con:06} would also be the fact that if every point on the boundary of $K$ is an apex of some cone inscribed in $K$ and circumscribed on $L$, then $K,L$ are concentric, homothetic ellipsoids, and moreover $L=-\frac{1}{n}K$.

\begin{figure}
\includegraphics[width=.5\textwidth]{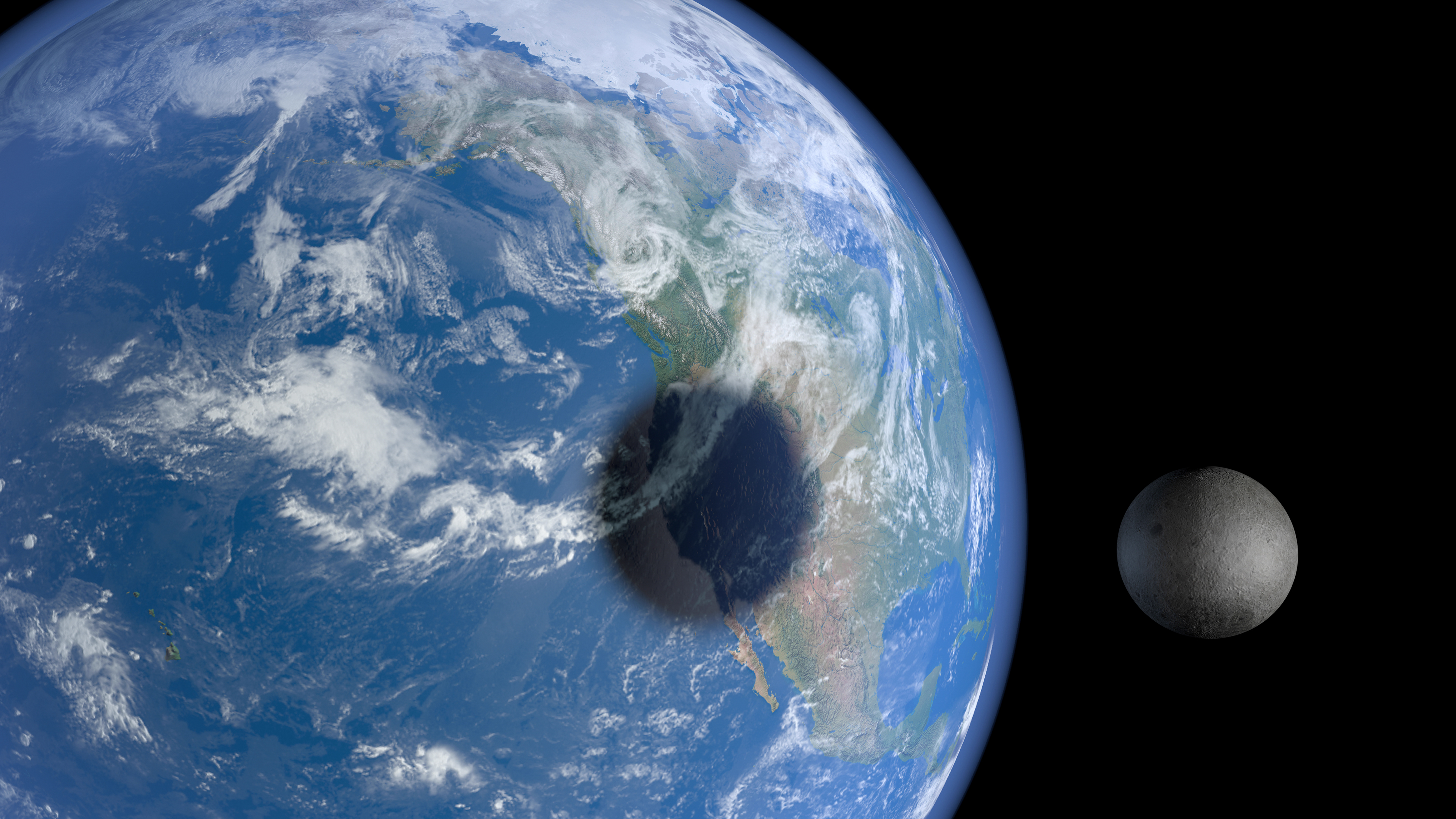}
\caption{The Earth illuminated by a point light source (the Sun), with the Moon projecting a shadow on its surface, representing the region of total solar eclipse. (Steven Molina/\protect\url{shutterstock.com})}
\label{fig:07}
\end{figure}

\subsection{sec:02}

The root of all trouble turns out to be the possibility of too many pairs of orthogonal vectors $\boldsymbol u_i$. In this case, instead of an additional constraint on vectors \eqref{eq:09}, we get a much less informative constraint on scalars \eqref{eq:10}. On the other hand, if the set of pairwise orthogonal vectors is sufficiently large, then the conclusion becomes true again:

\begin{lemma}\thlabel{lem:05}
Suppose that $G$ contains an independent set of size $n-1$. Then \eqref{eq:24} holds for some fixed $\boldsymbol w\in\mathbb R^{n-1}$ and every $u_i\in V(G)$.
\end{lemma}

\begin{proof}
Denote the subgraph induced by the vertices of the independent set by $H$, and let $\boldsymbol w\in\mathbb R^{n-1}$ be the unique vector satisfying $\langle\boldsymbol u_i,\boldsymbol w\rangle=\lambda_i$ for every $u_i\in V(H)$. Denote $\tilde{\boldsymbol v}_i\colonequals\boldsymbol v_i-\boldsymbol w$. Equation \eqref{eq:10} yields $\langle\boldsymbol u_i,\tilde{\boldsymbol v}_j\rangle=0$ for every pair of distinct vertices $u_i,u_j\in V(H)$. Hence $\tilde{\boldsymbol v}_i$ is orthogonal to the span of vectors corresponding to the remaining $n-2$ vertices of $H$, which in turn is the orthogonal complement of $\Span\{\boldsymbol u_i\}$. It follows that $\tilde{\boldsymbol v}_i=\mu_i\boldsymbol u_i$.\\

Now, let $u_i\notin V(H)$. By \thref{lem:03}, $u_i$ is connected to all the vertices of $H$. Let $u_1,u_2\in V(H)$. Equation \eqref{eq:09} yields
$$\tilde{\boldsymbol v}_i-\tilde{\boldsymbol v}_j=-\frac{\langle\boldsymbol u_j,\tilde{\boldsymbol v}_i\rangle}{\langle\boldsymbol u_j,\boldsymbol u_i\rangle}\boldsymbol u_i-\frac{\lambda_i-\langle\boldsymbol u_i,\boldsymbol v_j\rangle}{\langle\boldsymbol u_i,\boldsymbol u_j\rangle}\boldsymbol u_j,\quad j=1,2.$$
Thus $\tilde{\boldsymbol v}_i\in\Span\{\boldsymbol u_i,\boldsymbol u_1\}\cap\Span\{\boldsymbol u_i,\boldsymbol u_2\}=\Span\{\boldsymbol u_i\}$, which concludes the proof.
\end{proof}

Having established \thref{lem:05}, we could try to return to the question whether $6$ is an upper bound on $L(4)$. In the case where $G=K_{3,3}$, the graph does not admit a proper ear decomposition with ears of vertex-length at most $3$, but it does contain an independent set of size $3$, in which case, by \thref{lem:05}, we can proceed as before. Further, if $G\supsetneq K_{3,3}$ or $G\supseteq W_6$, then the graph already admits the desired ear decomposition. Unfortunately, if $G=Y_3$, then the graph admits neither the desired ear decomposition nor an independent set of size $3$, which prevents us from concluding the argument.\\

However, it seems to us that this obstacle may occur only on a nowhere dense set, unless the hypersurface is locally contained in a quadric. More precisely, we believe that the following conjecture holds:

\begin{conjecture}\thlabel{con:07}
Let $M\subset\mathbb R^n$, $n\geq 3$, be an open patch of a hypersurface of class $C^2$, and let $\Gamma_1,\Gamma_2$ be space curves. Suppose that for every point $p\in M$, there exist unique intersection points $u_1=\Gamma_1\cap T_pM,u_2=\Gamma_2\cap T_pM$ of the tangent hyperplane $T_pM$ with $\Gamma_1,\Gamma_2$, respectively, and that the vectors $u_1-p,u_2-p\in T_pM$ are orthogonal with respect to the second fundamental form $\mathrm{II}_p$ of $M$ at $p$. Then $M$ is contained in a quadric.
\end{conjecture}

\noindent Observe that the assumption of \thref{con:07} is invariant under projective automorphisms. Hence, if $M$ is contained in a quadric, then without loss of generality, we may assume that it is contained in the unit sphere, in which case we can easily verify that the straight lines
$$\Gamma_i=\langle\mathbf 1-2\hat{\boldsymbol e}_i\rangle+\hat{\boldsymbol e}_i-\hat{\boldsymbol e}_n,\quad i=1,2,\ldots,n-1,$$
enjoy the property that the vectors $u_i-p$, $i=1,2,\ldots,n-1$ are pairwise orthogonal for every point $p\in M$. However, it is the other implication that is the trickier one. Because the assumption becomes stronger with increasing dimension, we believe that two curves should be sufficient in an arbitrary dimension.

\subsection{sec:03}

An interesting question is whether \thref{lem:03} actually characterizes the family of all possible orthogonality graphs. In other words, is it true that every graph $G$ with the property that every induced subgraph of $G$ on $n$ vertices is connected, is an orthogonality graph of some finite subset of $\mathbb R^{n-1}$ of vectors in general position? Numerical evidence seems to support this hypothesis.\\

Although the example of a convex body with the property that $n$ transversal flat shadow boundaries pass through every point on its boundary seems quite rigid, it is relatively easy to construct a generic family of open patches of convex hypersurfaces with the property that $n-1$ transversal flat shadow boundaries pass through every point. Indeed, let $U\subseteq\mathbb R^n$ be any open subset of $\mathbb R^n$, let $H_1,H_2,\ldots,H_{n-1}:U\to\mathbb R$ be any submersions that induce foliations on $U$ where the leaves $H_1^{-1}(t),H_2^{-1}(t),\ldots,H_{n-1}^{-1}(t)$ are contained in transversal hyperplanes for every $t\in\mathbb R$, and let $p_1,p_2,\ldots,p_{n-1}:\mathbb R\to\mathbb R^n$ be any immersed curves. Define the vector field
$$\boldsymbol N(x)\colonequals(x-p_1(H_1(x)))\times(x-p_2(H_2(x)))\times\cdots\times(x-p_{n-1}(H_{n-1}(x))).$$
Now, if there exists a submersion $F:U\to\mathbb R$ such that $\nabla F(x)\parallel\boldsymbol N(x)$, then $F$ induces a foliation on $U$ where the leaves $F^{-1}(t)$ are hypersurfaces $L_t$ with the property that $n-1$ transversal flat shadow boundaries contained in $$H_1^{-1}(H_1(x)),H_2^{-1}(H_2(x)),\ldots,H_{n-1}^{-1}(H_{n-1}(x)),$$ created, respectively, by point light sources $$p_1(H_1(x)),p_2(H_2(x)),\ldots,p_{n-1}(H_{n-1}(x)),$$ pass through every point $x\in L_t$. By Frobenius' theorem, this is the case if and only if $$\omega\wedge\mathrm d\omega=0$$ on $U$, where $\omega\colonequals\boldsymbol N^\flat$ is the $1$-form dual to $\boldsymbol N$. Because the number of constraints (depending on derivatives of at most the first order) is significantly smaller than the number of degrees of freedom resulting from the choice of immersions, the family of solutions should be high-dimensional.

%%%%%%%%%%%%%%%%%%%%%%%%%%%%%%%%%%%%%%%%%%%%%%%%%%%%%%%%%%%%%%%%%%%%%%
\section*{Acknowledgments}
%%%%%%%%%%%%%%%%%%%%%%%%%%%%%%%%%%%%%%%%%%%%%%%%%%%%%%%%%%%%%%%%%%%%%%

I would like to thank Prof. Serhii Myroshnychenko for giving me this problem to consider, Prof. Efr\'en Morales-Amaya for bringing the supporting cones to my attention and many inspiring discussions during the Analysis and Convex Geometry Week at UniAndes, and my brother Dr. Michał Zawalski for an unwavering passion for combinatorics, countless hours of discussions (not only) about mathematics, and help in completing all the details.

\bibliography{references}{}
\bibliographystyle{amsplain}

\end{document}